\documentclass[11pt,a4paper]{article}

\pdfoutput=1

\usepackage{calc}
\usepackage{color}
\usepackage{amsmath}
\usepackage{amsthm}
\usepackage{amssymb}
\usepackage{graphicx}
\usepackage{natbib}
\usepackage{enumitem}

\usepackage[margin=2em,labelfont=bf]{caption}
\usepackage[english]{babel}

\renewcommand{\baselinestretch}{1.2}

\newtheorem{thm}{Theorem}
\newtheorem{lem}[thm]{Lemma}
\newtheorem{cor}[thm]{Corollary}

\newcommand{\rr}[1]{{\ttfamily\slshape\color{DarkGreen} #1}}

\definecolor{DarkBlue}{rgb}{0,0,0.5451}
\definecolor{DarkGreen}{rgb}{0,0.39216,0}
\definecolor{LightYellow}{rgb}{1,1,.8}
\definecolor{orange}{rgb}{.9,0.3445,0}

\makeatletter
\let\oldc\c
\renewcommand{\|}{|\!|}         
\newcommand{\T}{{}^{\mathsf{T}}}

\newcommand{\new}{\text{new}}
\newcommand{\ML}{\text{ML}}

\DeclareMathOperator*{\argmin}{argmin}

\DeclareMathOperator{\cov}{Cov}                   
                 
\DeclareMathOperator{\var}{Var}                   
\newcommand{\tr}{\text{tr}}
\DeclareMathOperator{\E}{E}                       

\newcommand{\cD}{{\cal{D}}} 
\newcommand{\cK}{{\cal{K}}}

\newcommand{\cO}{{\cal{O}}} 
\newcommand{\IN}{{\mathbb{N}}}
\newcommand{\IR}{{\mathbb{R}}}
\newcommand{\btheta}{{\boldsymbol{\theta}}{}}

\newcommand{\bsigma}{{\boldsymbol{\sigma}}}
\newcommand{\bTheta}{{\boldsymbol{\Theta}}}
\newcommand{\bSigma}{{\boldsymbol{\Sigma}}{}}
\newcommand{\A}{{\mathbf{A}}}
\newcommand{\B}{{\mathbf{B}}}
\newcommand{\C}{{\mathbf{C}}}
\newcommand{\D}{{\mathbf{D}}}
\newcommand{\F}{{\mathbf{F}}}
\newcommand{\G}{{\mathbf{G}}}
\newcommand{\I}{{\mathbf{I}}}
\newcommand{\K}{{\mathbf{K}}}
\newcommand{\bfM}{{\mathbf{M}}}
\newcommand{\N}{{\mathbf{N}}}
\newcommand{\bfT}{{\mathbf{T}}} 
\newcommand{\Z}{{\mathbf{Z}}}
\newcommand{\0}{{\mathbf{0}}}
\newcommand{\1}{{\mathbf{1}}}
\renewcommand{\a}{{\textbf{\textit{a}}}}
\newcommand{\h}{{\textbf{\textit{h}}}}
\renewcommand{\k}{{\textbf{\textit{k}}}}
\newcommand{\m}{{\textbf{\textit{m}}}}
\newcommand{\s}{{\textbf{\textit{s}}}}
\renewcommand{\v}{{\textbf{\textit{v}}}}
\newcommand{\w}{{\textbf{\textit{w}}}}
\newcommand{\x}{{\textbf{\textit{x}}}}
\newcommand{\z}{{\textbf{\textit{z}}}}

\renewcommand{\baselinestretch}{1.12}

\title{\vspace*{-13mm}\bf\LARGE Asymptotic properties
of multivariate \\[4mm] tapering for estimation and prediction }
\author{
{\begin{minipage}{0.34\textwidth}
\begin{center}
Reinhard Furrer \\
\begin{normalsize}
\textit{reinhard.furrer@math.uzh.ch} \\
University of Zurich \\~
\end{normalsize}
\end{center}
\end{minipage}}\hfill%
\begin{minipage}{0.36\textwidth}
\begin{center}
Fran\oldc cois Bachoc \\
\begin{normalsize}
\textit{francois.bachoc@univie.ac.at} \\
University of Vienna \\~
\end{normalsize}
\end{center}
\end{minipage}%
{\begin{minipage}{0.28\textwidth}
\begin{center}
Juan Du\\
\begin{normalsize}
\textit{dujuan@ksu.edu} \\
Kansas State University \\~
\end{normalsize}
\end{center}
\end{minipage}}
}
\begin{document}
\graphicspath{{figures/}{simfignew/}}

\maketitle

\noindent{\bf Abstract:} Parameter estimation for and prediction of
spatially or spatio--temporally correlated random processes are used
in many areas and often require the solution of a large linear system
based on the covariance matrix of the observations. In recent years,
the dataset sizes to which these methods are applied have steadily
increased such that straightforward statistical tools are
computationally too expensive to be used.  In the univariate context,
tapering, i.e., creating sparse approximate linear systems, has been
shown to be an efficient tool in both the estimation and prediction
settings. The asymptotic properties are derived under an infill
asymptotic setting.  In this paper we use a domain increasing
framework for estimation and prediction using multivariate tapering.
Under this asymptotic regime we prove that tapering (one-tapered form)
preserves the consistency of the untapered maximum likelihood
estimator and show that tapering has asymptotically the same mean
squared prediction error as using the corresponding untapered
predictor. The theoretical results are illustrated with simulations.


\section{Introduction}

Parameter estimation for and smoothing or interpolation of spatially
or spatio--temporally correlated random processes are used in many
areas and often require the solution of a large linear system based on
the covariance matrix of the observations. In recent years, the
dataset sizes to which these methods are applied have steadily
increased such that straightforward statistical tools are
computationally too expensive to be used.  For example, a typical
Landsat 7 satellite image consists of more than 34 million pixels
(30\,m resolution for an approximate scene size of
170\,km$\times$183\,km; source landsat.usgs.gov). Hence, classical
spatial and spatio--temporal models for such data sizes cannot be
handled with typical soft- and hardware. Thus, one typically relies on
approximation approaches.  In the univariate context, tapering,
i.e. creating sparse approximate linear systems through a direct
product of the (presumed) covariance function and a positive definite
but compactly supported correlation function, has been shown to be an
efficient tool in both the estimation and prediction settings.

The vast majority of the theoretical work on univariate tapering has
been placed in an infill--asymptotic setting using the concept of
Gaussian equivalent measures and mis-specified covariance functions
set forth in a series of papers by M. Stein
(\citeyear{Stei:88:AEP,Stei:90:UAO,Stei:97:ELP,Stei:99:PRF}).
Subsequently,
\cite{Furr:Gent:Nych:06,Kauf:Sche:Nych:08,Du:etal:09} and \cite{Wang:Loh:11} have
assumed a second-order stationary and isotropic Mat\'ern covariance to
show asymptotic optimality for prediction, consistency, and asymptotic
efficiency for estimation.  Recently, \cite{Stei:13} has extended these
results to other covariance functions by placing appropriate conditions
on the spectral density of the covariance.

In the infill--asymptotic setting, it is (essentially) sufficient to
match the degree of differentiability at the origin of an appropriately
chosen taper function with the smoothness of the (Mat\'ern) covariance
at the origin. Loosely speaking, for prediction, the predictor based
on tapered covariances has the same convergence rate as the optimal
predictor and the naive formula for the prediction kriging variance
has the correct convergence rate as well (Theorem~2.1 of
\citealp{Furr:Gent:Nych:06}, Theorem~1 of \citealp{Stei:13}).

For estimation, \citet{Kauf:Sche:Nych:08} introduced the concept of
one-taper and two-taper likelihood equations. In a one-taper setting
only the covariance is tapered while for two-tapered both the
covariance and empirical covariance are affected.  The one-taper
equation results in biased estimates while the two-taper equation is
an estimating equation approach and is thus unbiased. The price of
unbiased estimates is a (severe) loss of the computational efficiency
intended through tapering (see, e.g., Table~2 of
\citealp{Kauf:Sche:Nych:08} or Figure~2 of \citealp{Shab:Rupp:12}).

Extending the idea of tapering to a multivariate setting is not
straightforward.  The infill--asymptotic setting does not allow one to
`embed' the multivariate framework in a univariate one (e.g., as in
\citealp{Sain:Furr:Cres:11} for Gaussian Markov random fields).
\cite{Ruiz:Porc:15} introduced the concept of multivariate Gaussian
equivalent measures, but the conditions are difficult to verify and
their practical applicability is not entirely convincing. Several
authors have recently approached the problem using a increasing-domain
setting (\citealp{Shab:Rupp:12,Beli:etal:15}).  The main advantage of
this alternative sampling scheme is that we are not bound to Mat\'ern
type covariance functions nor to tapers that satisfy the taper
condition (i.e., sufficiently differentiable at the origin and at the
taper length). More so, we will show that for collocated data, other
practical tapers can be described.  The main disadvantage is the
somewhat less-intuitive conceptual framework.  For example, in the
case of heavy metal contents in sediments of a lake,
infill--asymptotics can be mimicked by taking more and more
measurements. In a increasing-domain setting, this is not
possible. On the other hand asymptotics is a theoretical concept and
in practice only a finite number of observations are available.

The main contributions of this paper are as follows: (i) under weak
conditions on the covariance matrix function and the taper (matrix)
function form we show that in a increasing-domain framework the
tapered maximum likelihood estimator preserves the consistency of the
untapered likelihood estimator; (ii) the difference between the
(integrated) mean squared prediction error of the tapered and the
untapered converges in probability to zero, even when prediction is
based on estimated parameters.  Note that although we require that the
taper range increases, no rate assumption is necessary; (iii)
numerical simulations illustrate that the approach has very appealing
finite sample properties, especially for prediction with plugin
estimates we find only a very small loss in efficiency.

This paper is structured as follows: Section~\ref{sec:notation}
introduces basic notation and relevant definitions.  The main results
are given in Section~\ref{sec:results}. Section~\ref{sec:ill}
illustrates the methodology using an extensive simulation study.
Concluding remarks are given in Section~\ref{sec:out}. Proofs and
technical results are presented in the appendix.

Note that compared with directly using compactly supported covariance
functions, tapering has several advantages. Our modeling experience
has shown that the (practical) dependence structure is often larger or
much larger than what can be handled computationally and additional
approximations would be needed anyway.  We see tapering as a
computational approximation that does not alter the statistical
model. The taper range (degree of tapering) depends on the
availability of memory and computing power and thus changes when the
analysis is carried out on different computers or at some later time
with improved hardware.

\section{Notation and setting}\label{sec:notation}
\newtheorem{cond}{Condition}

We denote (deterministic) vectors and matrices with bold lower and upper case
symbols. Random variables and processes are denoted with upper case
symbols and random vectors and vector processes are denoted with bold
upper case symbols.  For $\x \in \IR^m$, we let $|\x| =
\max_{i=1,\dots,m}|x_i|$ and $\|\x\| = \sqrt{ \sum_{i=1}^m x_i^2 }$.

The singular values of a $n\times n$ real matrix $\A=(a_{ij})$ are
denoted by $\rho_1(\A)\geq\dots \geq\rho_n(\A)\geq 0 $ and, in the
case when $\A$ is symmetric, the eigenvalues are denoted by
$\lambda_1(\A)\geq\dots \geq\lambda_n(\A)$.  The spectral norm is
given by $\rho_1(\A)$ and $\|\A\|_F^2=\sum_{i,j}|a_{ij}|^2 $ denotes
the Frobenius norm.

For a sequence of random variables $X_n$, we write $X_n = o_p(1)$ when
$X_n$ converges to $0$ in probability as $n \to \infty$ and we write
$X_n = O_p(1)$ when $X_n$ is bounded in probability as $n \to \infty$.
 
\bigskip

Let, for $d\in\IN^+$ and $p\in\IN^+$, fixed throughout this paper, 
\begin{align}
  \bigl\{ Z_k(\s): \s\in\cD\subset\IR^d,   k=1,\dots,p  \bigr\}
\label{eq:process}
\end{align}
be a multivariate stationary Gaussian random process. We let
$\Z(\s)=(Z_1(\s),\dots,Z_p(\s))\T$.  To simplify the notations, we
assume, essentially without loss of generality, that:

\begin{cond}\label{cond:process}
  Process \eqref{eq:process} has zero mean.
\end{cond}

Let $q \in \IN^+$ and let $\Theta$ be the compact subset
$[\theta_{\inf},\theta_{\sup}]^q$ with $- \infty < \theta_{\inf} <
\theta_{\sup} < + \infty$.  For each $\btheta \in \Theta$ we consider
a candidate stationary matrix covariance function for the
process~\eqref{eq:process}, of the form
$\C(\h;\btheta)=\big(c_{kl}(\h;\btheta)\big)$.  We assume that
there exists $\btheta_0 \in \Theta$, with for $i=1,\dots,q$,
$\theta_{\inf} < \theta_{0i} < \theta_{\sup}$, so that
$\C(\h;\btheta_0)= \cov\big(\Z(\s),\Z(\s+\h)\big)$.  The covariance
function $c_{kk}(\h;\btheta_0)$ of the $k$th (marginal) process is
called a direct covariance (function) and the off-diagonal elements
$c_{kl}(\h;\btheta_0)$, $k\neq l$, are called cross covariance
(functions).  We also consider a stationary taper matrix function of
the form $\big(t_{kl}(\h)\big)$, with $t_{kl}(\h)=0$ for $\|\h\|\geq
1$.

For any $n \in \IN^+$, the Gaussian processes~\eqref{eq:process} are observed
at the points $\x_1,\dots,\x_n \in \IR^d$:  

\begin{cond}\label{cond:collocated} 
   We dispose collocated observations at the distinct locations
  $\x_1,\dots,\x_n\in \IR^d$.
\end{cond}

For $i = (k-1)n + a$ and
$j = (l-1) n + b$, with $k,l=1,\dots,p$ and $a,b=1,\dots,n$, we let
$\z$ be the $np\times1$ Gaussian vector with $z_i=Z_k(\x_a)$, for
$\btheta \in \Theta$ we let $\bSigma_{\btheta}$ be the $np \times np$
covariance matrix with $\sigma_{\btheta i j} = c_{kl}(\x_a-\x_b ;
\btheta)$ and $\bfT$ be the $np \times np$ taper covariance matrix
with $t_{ij} = t_{kl}\big( (\x_a-\x_b)/\gamma_n \big)$, where
$\gamma_n>0$ is the taper range.  We let $\K_{\btheta} =
\bSigma_{\btheta} \circ \bfT$, where the symbol $\circ$ denotes the
direct (Schur) product.
 
The maximum likelihood  (ML) estimator is defined by
$\hat{\btheta}_{\ML} \in \argmin_{\btheta} L_{\btheta}$, with
\begin{equation} \label{eq:MLtheta} L_{\btheta} = \frac{1}{np}
  \log{\left( \det{\left( \bSigma_{\btheta} \right)} \right) } +
  \frac{1}{np} \z \T \bSigma_{\btheta}^{-1} \z.
\end{equation}
The tapered ML estimator is defined by
$\hat{\btheta}_{t\ML} \in \argmin_{\btheta} \bar{L}_{\btheta}$, with
\begin{equation} \label{eq:tMLtheta} \bar{L}_{\btheta} = \frac{1}{np}
  \log{\left( \det{\left( \K_{\btheta} \right)} \right) } +
  \frac{1}{np} \z \T \K_{\btheta}^{-1} \z.
\end{equation}

We can assume, without loss of generality, that $Z_1(\x)$ is the
Gaussian process that is predicted at new points.  Then, for $\x \in
\IR^d$, let $\bsigma_{\btheta}(\x)$ be the $np \times 1$ vector
defined by, for $i = (k-1)n + a$, $k=1,\dots,p$, $a=1,\dots,n$,
$\sigma_{\btheta}(\x)_i = c_{1k}(\x - \x_a ; \btheta)$.  Define
similarly the $np \times 1$ vector $\k_{\btheta}(\x)$ by
$k_{\btheta}(\x)_i = c_{1k}(\x - \x_a ; \btheta) t_{1k}\big(
(\x - \x_a)/ \gamma_n \big)$.

\section{Consistent estimation and asymptotically equal prediction}\label{sec:results}

We first explore four conditions on covariance and taper matrix functions.
The following condition holds for all the most classical models of covariance functions with infinite supports. Note that models with compactly supported covariance functions can be non-differentiable with respect to the covariance parameters, but that tapering is irrelevant anyway in increasing-domain asymptotics when the original covariance functions are already compactly supported.
 
\begin{cond}  \label{cond:ctheta:one:derivative}
For all fixed $\x \in \IR^d$, $k,l=1,\dots,p$, $c_{kl}(\x;\btheta)$ is continuously differentiable with respect to $\btheta$.
There exist constants $A < + \infty$ and $\alpha >0$ so that for all $i =1,\dots,q$, for all $\x \in \IR^d$ and for all $\btheta \in \Theta$,
\begin{equation*} 
\left|  c_{kl}\left( \x ; \btheta \right) \right| \leq \frac{A}{ 1+|\x|^{d+\alpha} }
~ ~ ~ ~
\mbox{and}
~ ~ ~ ~
\left| \frac{\partial}{\partial \theta_i}  c_{kl}\left( \x ; \btheta \right) \right| \leq \frac{A}{ 1+|\x|^{d+\alpha} }.
\end{equation*}
\end{cond}

\begin{cond}  \label{cond:taper:no:rate}
For all $k,l=1,\dots,p$, the taper function $t_{kl}$ is continuous at $\0$ and satisfies $t_{kl}(\0) = 1$ and $|t_{kl}(\x)|\leq 1$ for all $\x \in \IR^d$.
The taper range $\gamma = \gamma_n$ satisfies $\gamma_n \to_{n \to \infty} + \infty$.
\end{cond}

The next condition on a minimal distance between two different observation points is assumed in most domain increasing settings.

\begin{cond}  \label{cond:minimal:distance}
There exists a constant $\Delta >0$ so that for all $n \in \IN^+$ and for all $a \neq b$, $|\x_a - \x_b| \geq \Delta$.
\end{cond}

\begin{cond}  \label{cond:minimal:eigenvalue}
There exists a constant $\delta >0$ so that for all $n \in \IN^+$ and for all $\btheta \in \Theta$, $\lambda_{np}( \bSigma_{\btheta} ) \geq \delta$ and $\lambda_{np}( \K_{\btheta} ) \geq \delta$.
\end{cond}

We expect Condition~\ref{cond:minimal:eigenvalue} to hold in many cases when Condition~\ref{cond:minimal:distance} also holds. For univariate tapering, Condition~\ref{cond:minimal:eigenvalue} would indeed hold under mild assumptions (consider an adaptation of Proposition D.4 in \citealp{bachoc14asymptotic}). Furthermore, when the parametric model incorporates a nugget effect or measurement errors, then Condition~\ref{cond:minimal:eigenvalue} holds provided that the nugget or error variances are lower-bounded uniformly in $\btheta$. The nugget or measurement error case is directly treated by Theorem~\ref{thm:consistent:estimation}; Theorem~\ref{thm:convergence:IMSE} would also be valid for it with a minor change of notation to define the integrated prediction errors (see, e.g., the context of \citealp{bachoc14asymptoticb}).

The next theorem and corollary (the corollary is proved using standard $M$-estimator techniques),  show that if the standard conditions for consistency of the (untapered) ML estimator hold, then the tapering preserves this consistency, as long as $\gamma \to_{n \to \infty} + \infty$.

\begin{thm} \label{thm:consistent:estimation}
Assume that Conditions~\ref{cond:ctheta:one:derivative},~\ref{cond:taper:no:rate},~\ref{cond:minimal:distance}, and~\ref{cond:minimal:eigenvalue} hold.
Then, as $n \to \infty$,
\[
\sup_{\btheta \in \Theta} | L_{\btheta} - \bar{L}_{\btheta} | = o_p(1).
\]
\end{thm}

\begin{cor} \label{cor:consistent:estimation}
Consider the same setting as in Theorem~\ref{thm:consistent:estimation}.
Assume that for all $\kappa >0$ there exists $\epsilon > 0$ so that
\[
\inf_{ | \btheta - \btheta_0 | \geq \kappa} L_{\btheta} - L_{\btheta_0} \geq \epsilon + o_p(1),
\]
where the $o_p(1)$ may depend on $\epsilon$ and $\kappa$ and goes to $0$ in probability as $n \to \infty$.
Then, as $n \to \infty$,
\[
\widehat{\btheta}_{\ML} \to_p \btheta_0
~ ~ ~ ~ ~ ~
\mbox{and}
~ ~ ~ ~ ~ ~
\widehat{\btheta}_{t\ML} \to_p \btheta_0.
\]
\end{cor}

Theorem~\ref{thm:consistent:estimation} and Corollary~\ref{cor:consistent:estimation} highlight the important difference between one-taper and two-taper ML in terms of asymptotics. One-taper approximation with fixed range $\gamma$ and independent of $n$ boils down to an incorrectly specified covariance model. Thus, with fixed~$\gamma$, the tapered ML estimator would generally be inconsistent and would converge to the asymptotic minimizer of a Kullback--Leibler divergence (for the univariate case, see the discussion in \citealp{Kauf:Sche:Nych:08}, and also \citealp{watkins90maximum}, or \citealp{bachoc14asymptoticb}). Hence, assuming $\gamma \to \infty$ is necessary to prove consistency, which we do here. Note that, nevertheless, no rate needs to be specified.
These facts also entail an exposition benefit for our paper: we simply have to show that the one-taper approximation does not damage the untapered ML estimator. The question of the consistency of this latter estimator can be treated in separate references, like \cite{mardia84maximum} or \cite{bachoc14asymptotic} for the univariate case. Especially, identifiability assumptions for the covariance model need not be discussed in our paper. 

On the other hand, for the two-taper ML, consistency can be proved for a fixed $\gamma$, provided notably that the model of tapered covariance and cross-covariance functions is identifiable. (In particular, two different covariance parameters yield two different sets of tapered covariance and cross-covariance functions.)
We refer to \cite{Shab:Rupp:12} for a corresponding proof in the univariate case. (Actually, we believe that a global identifiability condition might be missing in \cite{Shab:Rupp:12}, stronger than assumption (B) in this reference, for it is not clear how to go from (S.29) to (S.30) in its supplementary material.)
Hence, the difference between the asymptotic analysis of the untapered and two-taper ML estimators is more pronounced, since the latter estimator is a quasi-likelihood estimator in a covariance model different from the original one. This is why, in \cite{Shab:Rupp:12}, many assumptions, notably on identifiability, are restated independently of the untapered ML estimator.

These asymptotic considerations also correspond to practical aspects of the comparison between one- and two-taper equations. The latter can be employed with a smaller range $\gamma$ than the former, which is beneficial, but on the other hand, requires the full inverse of a sparse matrix.

The following theorem shows that tapering has no asymptotic effect on prediction, uniformly in the covariance parameter $\btheta$. (Note that for prediction, there is no distinction between one and two-taper approximation.)

\begin{thm} \label{thm:convergence:IMSE}
Assume that Conditions~\ref{cond:ctheta:one:derivative},~\ref{cond:taper:no:rate},~\ref{cond:minimal:distance}, and~\ref{cond:minimal:eigenvalue} hold. Let $(\x_{\new,n})_{n \in \IN^+}$ be a fixed sequence in $\IR^d$.
Then, as $n \to \infty$,
\begin{equation} \label{eq:convergence:fixed:MSE}
\sup_{\btheta\in\Theta}
\left| 
 \left[ \bsigma_{\btheta}(\x_{\new,n}) \T \bSigma_{\btheta}^{-1} \z - Z_1(\x_{\new,n}) \right]^2 
-
 \left[ \k_{\btheta}(\x_{\new,n}) \T \K_{\btheta}^{-1} \z - Z_1(\x_{\new,n}) \right]^2  
\right| = o_p(1).
\end{equation}

Assume furthermore that for any fixed $\btheta$, $k$ and $l$, the functions $c_{kl}(\x;\btheta)$ and $t_{kl}(\x)$ are continuous.
Let $\cD_n$ be a sequence of measurable subsets of $\IR^d$ with positive Lebesgue measures and let $f_n(\x)$ be a sequence of continuous probability density functions on $\cD_n$.
Then, as $n \to \infty$,
\begin{equation} \label{eq:convergence:IMSE}
\sup_{\btheta\in\Theta}
\left| 
\int_{\cD_n} \left[ \bsigma_{\btheta}(\x) \T \bSigma_{\btheta}^{-1} \z - Z_1(\x) \right]^2  f_n(\x) d \x
-
\int_{\cD_n} \left[ \k_{\btheta}(\x) \T \K_{\btheta}^{-1} \z - Z_1(\x) \right]^2  f_n(\x) d \x
\right| = o_p(1).
\end{equation}
\end{thm}

In \eqref{eq:convergence:IMSE}, we assume continuity of the cross covariance, covariance and taper functions, and of $f_n(\x)$ in order to define integrals in the $L^2$ sense.
When $f_n(\x)$ is constant on $\cD_n$, Theorem~\ref{thm:convergence:IMSE} shows that tapering does not damage the mean integrated square prediction error over any sequence of prediction domains $\cD_n$.
Furthermore, in \eqref{eq:convergence:fixed:MSE} and \eqref{eq:convergence:IMSE}, the terms in the differences are typically bounded away from zero in probability, because of Condition~\ref{cond:minimal:distance} (consider for example Equation (10) in Proposition 5.2 of \citealp{bachoc14asymptotic}). (This would not hold only in degenerate cases when $\x_{\new,n}$ becomes arbitrarily close to an observation point or where $f_{n}(\x)$ concentrates around an observation point.) Hence, also the ratio of (integrated) mean square prediction errors, between tapered and untapered predictions, converges to unity in general. Finally, because of the supremum over $\btheta$ in \eqref{eq:convergence:fixed:MSE} and \eqref{eq:convergence:IMSE}, Theorem~\ref{thm:convergence:IMSE} implies that the difference of tapered and untapered prediction errors goes to zero also when the predictions are obtained from any common estimator~$\widehat{\btheta}$.

\bigskip

{\bf Remark:} The condition $t_{kl}(\boldsymbol{0}) = 1$ in Condition~\ref{cond:taper:no:rate} is necessary for Theorem~\ref{thm:consistent:estimation}. Indeed, it is typically needed in order to guarantee that $1/(np)  \| \bSigma_{\btheta} - \K_{\btheta} \|^2_F$ goes to zero. The latter is necessary for Theorem~\ref{thm:consistent:estimation}, as can be shown from the arguments in the proof of Proposition~3.1 in \cite{bachoc14asymptotic}. The  condition  $t_{kl}(\boldsymbol{0}) = 1$ should also be needed for Theorem~\ref{thm:convergence:IMSE}, as is suggested by the second offline equation in Proposition 5.1 in \cite{bachoc14asymptotic}.
 
\section{Simulations and illustrations}\label{sec:ill}

We now evaluate the finite sample performance of multivariate
tapering with simulations.  We consider a bivariate Gaussian
isotropic process with Mat\'ern type direct and cross-covariances
\begin{equation}
  c_{kl}(\x;\btheta)=  \frac{\sigma_{kl}^2}{2^{\nu_{kl}-1}\Gamma(\nu_{kl})}
  ( \|\x\| /\rho_{kl})^{\nu_{kl}}\cK_{\nu_{kl}}(\|\x\|/\rho_{kl}), \qquad k,l=1,2
  \label{eq:mat.cov}
\end{equation}
where $\Gamma$ is the Gamma function and $\cK_\nu$ is the modified Bessel
function of the second kind of order $\nu$ \citep{Abra:Steg:70}.  To
ensure positive definiteness, constraints on $\{
\sigma_{kl},\,\rho_{kl},\,\nu_{kl}, k,l=1,2 \}$ have to be imposed,
see \cite{Gnei:Klei:Schl:10}.  We use two different covariance models:
\begin{enumerate}[label={(\Alph*)}]
\item\label{A} ranges:~~~~~~~ $\rho_{11}=5$,  $\rho_{12}=3$, $\rho_{22}=4$\\
  sills:~~~~~~~~~~ $\sigma_{11}=1$,  $\sigma_{12}=.6$, $\sigma_{22}=1$ \\
   smoothness: $\nu_{11}=\nu_{12}=\nu_{22}=1/2$
\item\label{B}  ranges:~~~~~~~  $\rho_{11}=3$,  $\rho_{12}=3$, $\rho_{22}=4$\\
  sills:~~~~~~~~~~   $\sigma_{11}=1$,  $\sigma_{12}=.7$, $\sigma_{22}=1$ \\
   smoothness: $\nu_{11}=3/2$,  $\nu_{12}=1$, $\nu_{22}=1/2$
\end{enumerate}
The smoothness parameters will not be estimated and are fixed. Hence,
$\btheta=(\rho_{11},\rho_{12},\rho_{22},\sigma_{11},$
$\sigma_{12},\sigma_{22})\T$ and $q=6$.  The Mat\'ern covariance
functions satisfy Condition~\ref{cond:ctheta:one:derivative}.

\begin{figure}[ht]
  \centering
  \includegraphics[width=.9\textwidth]{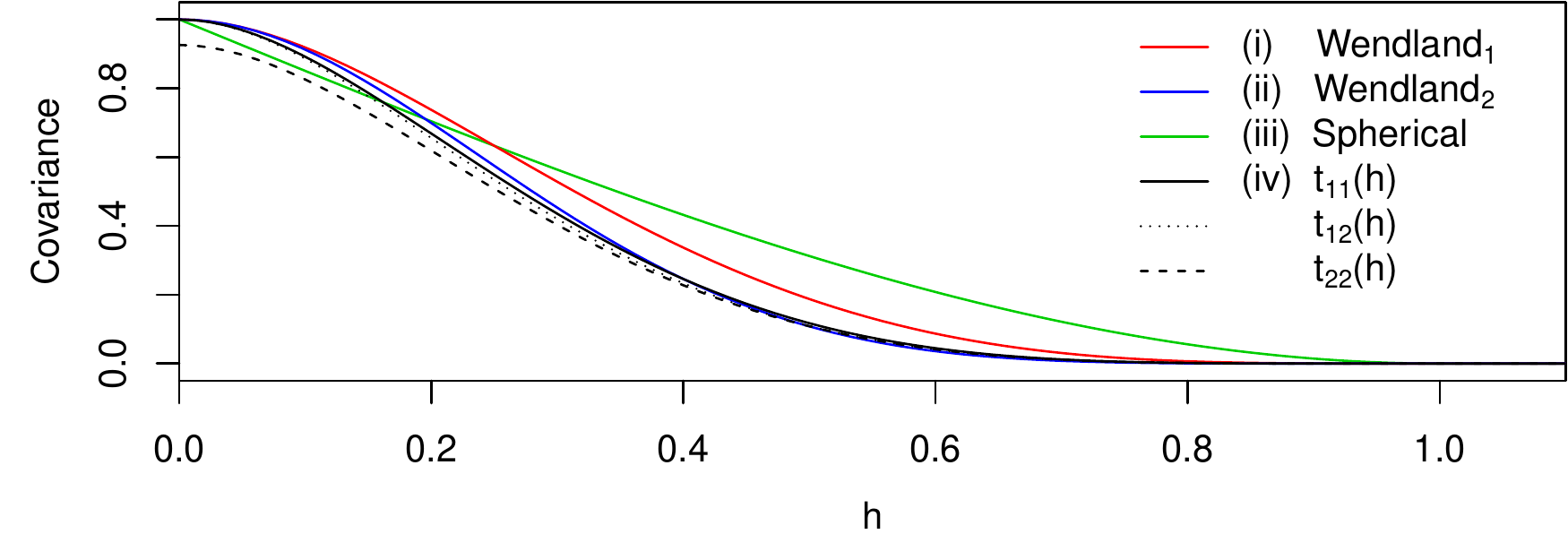}
  \caption{Different taper functions.}\label{covs}
\end{figure}

We consider the following taper matrix functions:
\begin{enumerate}[label={(\roman*)}]
\item\label{i} $t_{kl}(\x)=(1 - \|\x\|)^4_+ (1+4\|\x\|)$, ~$k,l=1,2$. 
\item\label{ii} $t_{kl}(\x)=(1 - \|\x\|)^6_+ (1+6\|\x\|+35\|\x\|^2/3)$, ~$k,l=1,2$. 
\item\label{iii} $t_{kl}(\x)=(1 - \|\x\|)^2_+ (1+\|\x\|/2)$, ~$k,l=1,2$. 
\item\label{iv} $ t_{11} (\x) = (1-\|\x\|)_+^{5}(1+5\|\x\|+\|\x\|^2)$,
  $t_{12} (\x) = t_{21} (\x) = \sqrt{6/7} \,(1-\|\x\|)_+^{5}
  (1+5\|\x\|+\|\x\|^2)$, $t_{22} (\x) = (1-\|\x\|)_+^{5}(1+5\|\x\|)$.
\end{enumerate}

Taper matrix functions~\ref{i}--\ref{iii} satisfy
Condition~\ref{cond:taper:no:rate} and the associated taper matrices
are of the form $\bfT=\1\1\T\otimes t\bigl(\|\x_a-\x_b\|/\gamma\bigr
)$ where the symbol $\otimes$ denotes the Kronecker product and where
$t(\cdot)$ is as indicated above. In the literature these functions
are referred to as Wendland$_1$, Wendland$_2$ and spherical taper
\citep{Wend:95,Furr:Gent:Nych:06}.

Taper matrix function~\ref{iv} is taken from \cite{Deme:13} Corollary
2.2.3, based upon results from Theorem 3 of \cite{Ma:11b} and Lemma 2 of
\cite{Ma:11c}. The validity of this taper matrix function can also be
shown using Theorem~A in \cite{Dale:Porc:Bevi:14} published later.
Taper matrix function~\ref{iv} has $t_{12}(\0)=\sqrt{6/7}<1$ (see
Figure~\ref{covs}) and we investigate its finite sample behavior
although Condition~\ref{cond:taper:no:rate} is violated.  We expect
similar behavior of~\ref{i}, \ref{ii}, and~\ref{iv} as the (direct)
taper functions are very similar.

\begin{figure}[b]
  \centering
  \includegraphics[width=.5\textwidth]{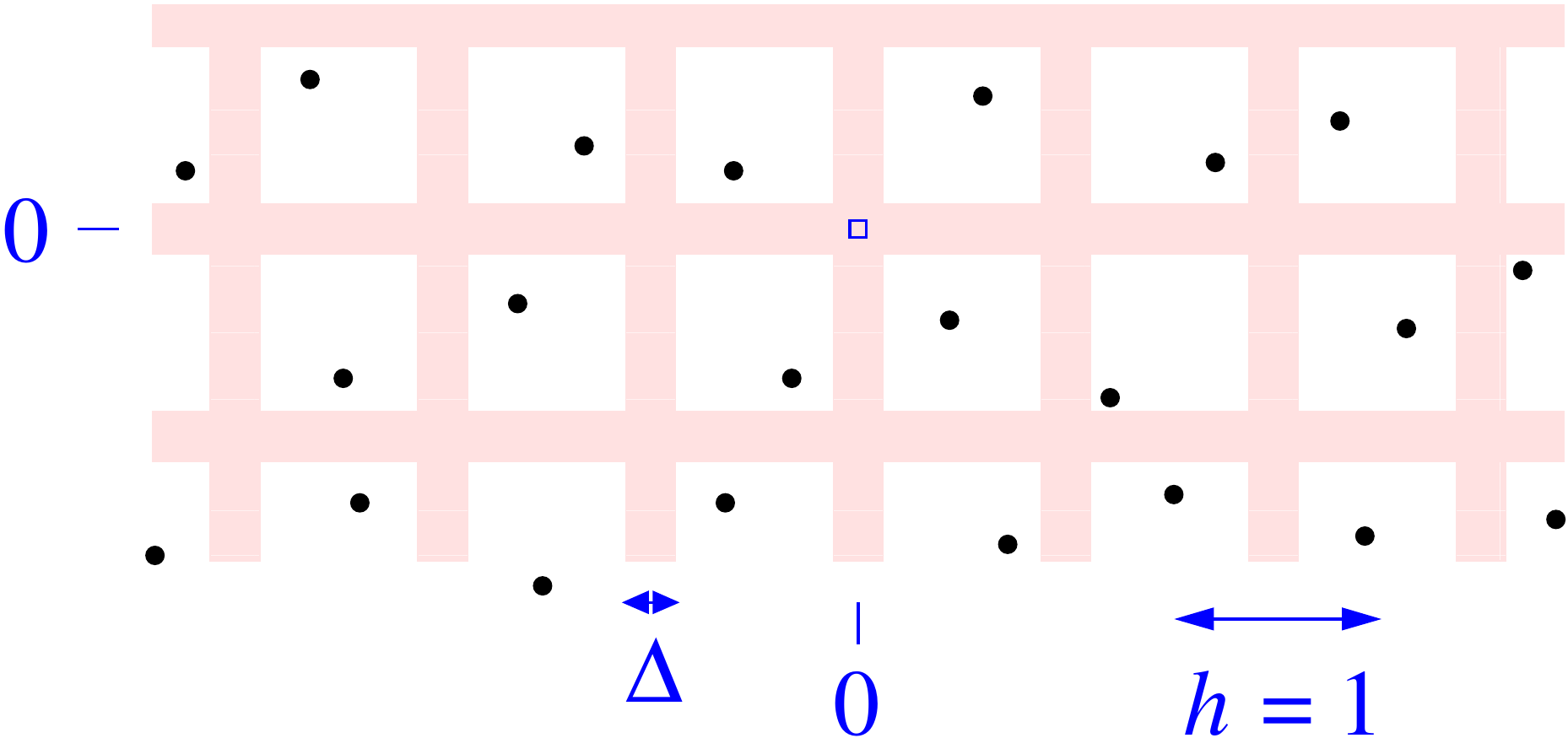}%
  \caption{One set of sampled locations with simulation parameter
    $\Delta=0.2$ and square center spacing $h=1$.}\label{locations}
\end{figure}

We are sampling $4m^2$ locations uniformly in a domain defined by the
union of squares $[(1-\Delta)/2]^2$, centered at
$\{\pm(r-1/2),\pm(s-1/2)\}$, $r,s=1,m$. The parameter $\Delta$
represents the minimum distance between the locations and the case
$\Delta=1$ is a regular grid. Prediction is done at the location
$\x_{\new}=(0,0)\T$ in the center of the
domain. Figure~\ref{locations} illustrates the setup.  We present
results for the two cases $\Delta=0.2,1$ (thus satisfying
Condition~\ref{cond:minimal:distance}) and three grid size parameter
values $m=10,16,25$, i.e., $n=400,1024,2500$ and covariance matrix
sizes $800\times800$, $2048\times 2048$, $5000\times 5000$,
respectively.  Condition~\ref{cond:minimal:eigenvalue} has been
verified numerically.

The next two subsections discuss the results of estimation and
prediction.  Computational details are given in the last subsection.

\subsection{Estimation}
\begin{figure}[!h]
\vspace*{-4mm}
\includegraphics[width=\textwidth]{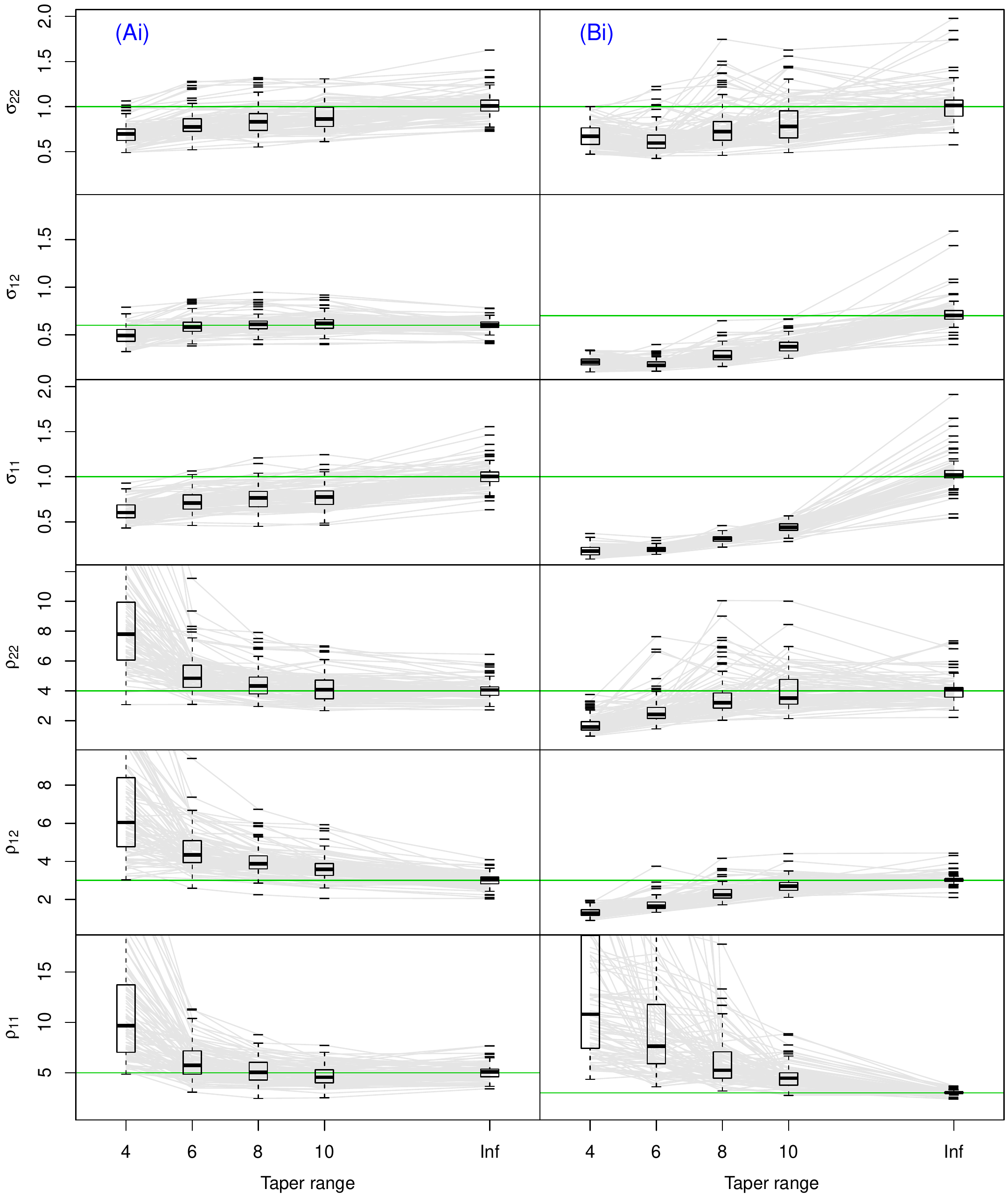}
\caption{Effect of increasing the taper range $\gamma$ on the ML
  estimates. Columns are for the two different covariance models, rows
  are for different parameters (truth is indicated by the horizontal
  green line). 100 realizations have been generated ($\Delta=1$) based
  on $n=400$. Each individual realization is indicated with a gray
  line.}
  \label{fig:taprange}
\end{figure}

We first investigate $\widehat{\btheta}_{t\ML}$ and compare it to
$\btheta_0$ as the taper range increases.  Figure~\ref{fig:taprange}
summarizes the estimates of $\widehat\btheta_{t\ML}$ for equispaced
observations ($\Delta=1$) with $n=400$, taper function~\ref{i}, and
using taper ranges $\gamma=4,6,8,10$ as well as no tapering
($\gamma=$~Inf). As expected, for small taper ranges the results are
biased with range parameters typically overestimated and sill
parameters underestimated. For smoother spatial fields~\ref{B}, the
bias and uncertainties are (slightly) larger.  The estimates of the
sill parameters benefit from a regularizing aspect of tapering and
thus exhibit a consistently smaller variance compared with the untapered
estimates.  This effect of regularizing is surprisingly strong for
model~\ref{B} and parameter $\sigma_{11}$.

Figure~\ref{fig:conv} shows the effect of increasing the number of
locations where we have added the boxplots for $n=1024$ and $n=2500$
(i.e., $m=16$ and $m=25$) to four panels of Figure~\ref{fig:taprange}.
For the untapered estimates, one clearly sees that the uncertainties
in the estimates decrease with increasing $n$. For the tapered
estimates this effect is not as pronounced because of the regularizing
effect of the tapering. As expected, the bias itself is not reduced by
increasing the number of observations while keeping the taper range
fixed. On the other hand, as illustrated in
Corollary~\ref{cor:consistent:estimation}, when going from
$n=400,\gamma=4$ to $n=2500,\gamma=10$, the distribution of the
tapered ML estimates becomes closer to that of the untapered ones.

\begin{figure}[!t]
\includegraphics[width=\textwidth]{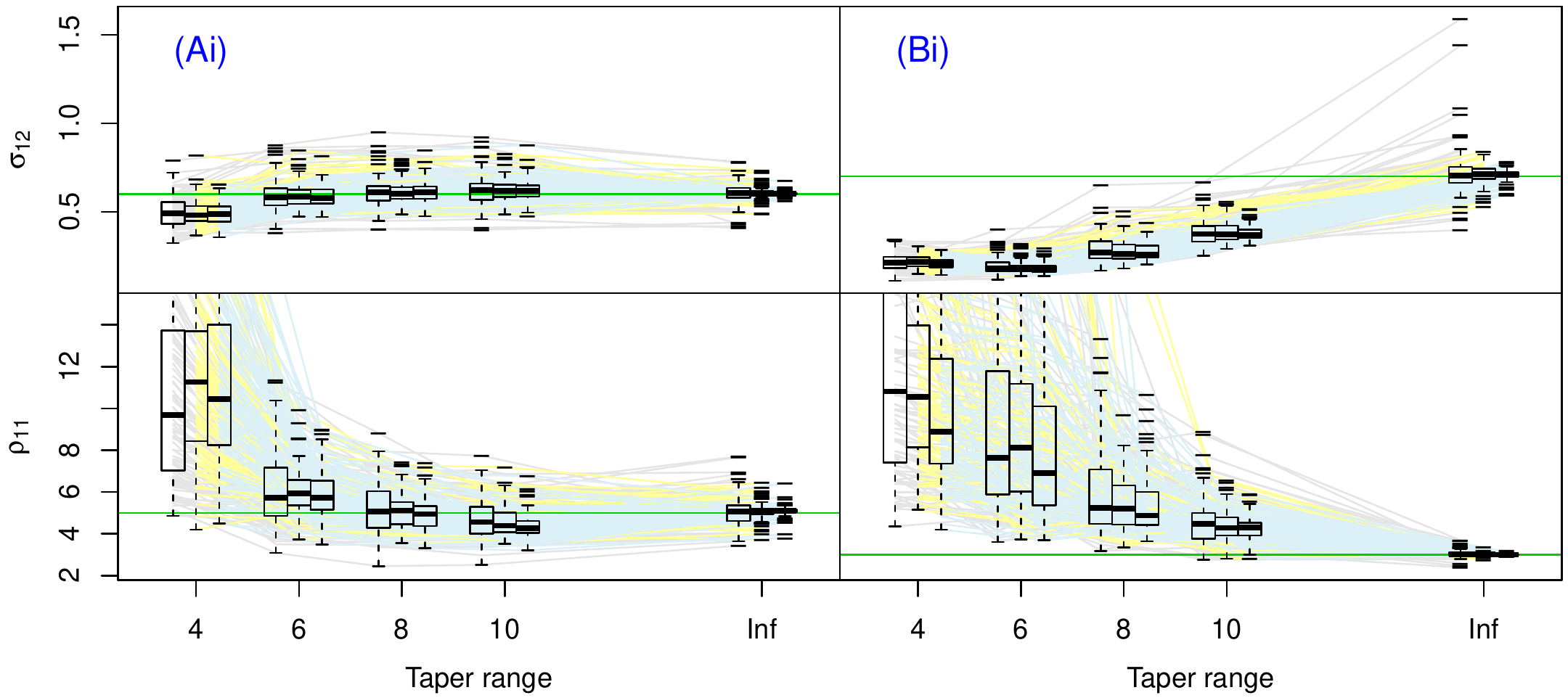}
\caption{Effect of increasing the domain on the ML estimates. The
  boxplots correspond to $n=400$ (gray), 1024 (yellow), 2500 (light
  blue), left to right for each taper range, $\Delta=1$. See also
  Figure~\ref{fig:taprange}.}
  \label{fig:conv}
\end{figure}

\subsection{Prediction}
In practice, prediction is often of prime interest and we primarily investigate
the effect of tapering on the prediction of the first process $Z_1$ at the
unobserved location $\x_{\new}=(0,0)\T$. As parameter values we use
$\btheta_0$ and $\widehat\btheta_{t\ML}$ for different taper ranges
$\gamma$.  

\begin{figure}[ht]
  \centering
\includegraphics[width=1.01\textwidth]{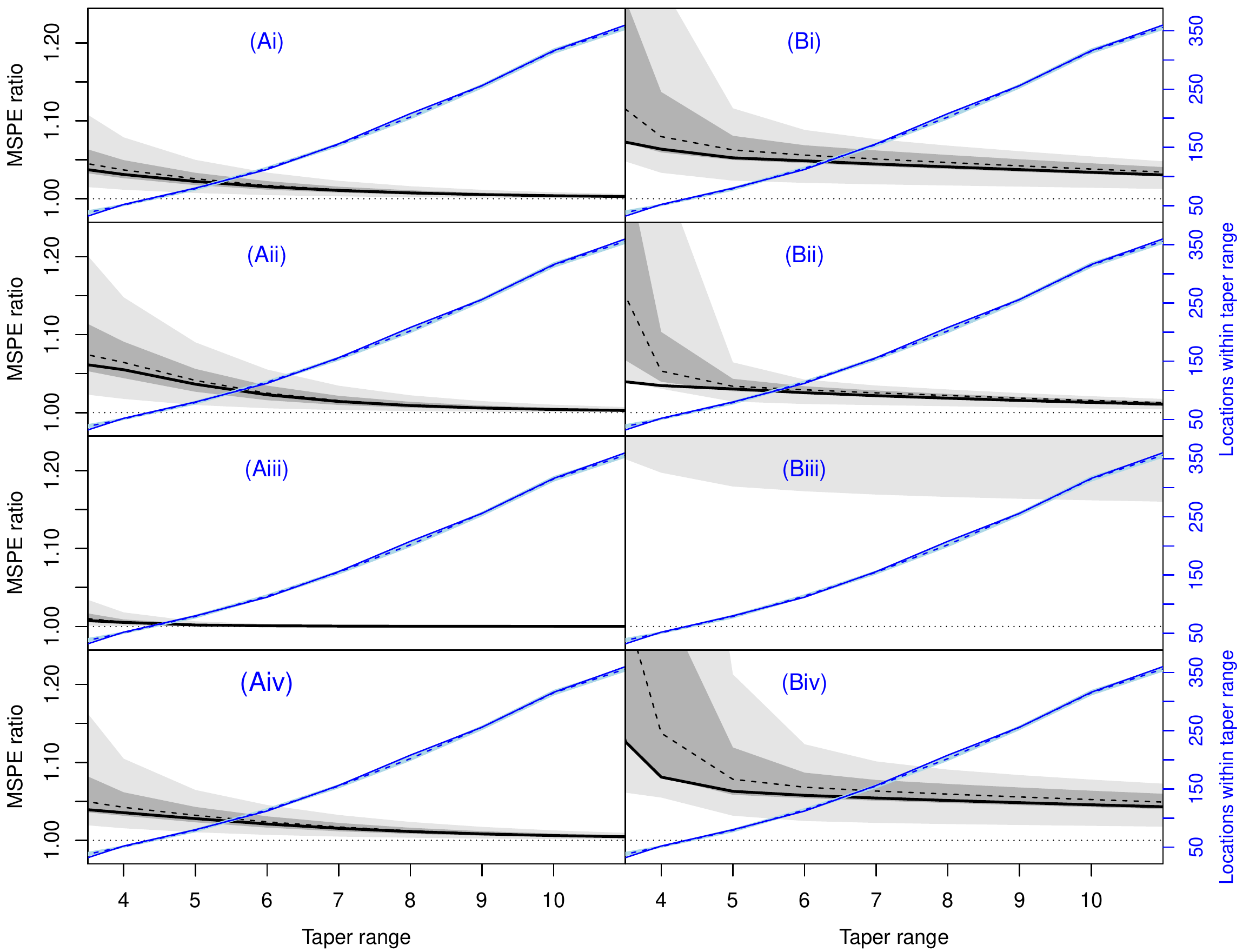}
\caption{Ratios of the tapered to the untapered MSPEs for $n=400$ using
  $\btheta_0$. The solid line represents MSPE ratios for equispaced
  locations ($\Delta=1$), the dashed line shows the median MSPE ratios
  from 100 simulations with random locations with $\Delta=0.2$ (gray
  and light gray are pointwise 50 and 95 percentiles). The blue lines
  indicate the number of points within the taper range (mean solid,
  median dashed and light blue pointwise 95 percentiles). }
  \label{fig:msper}
\end{figure}

In Figure~\ref{fig:msper} we display the ratio of the tapered to the
untapered mean squared prediction errors (MSPEs) using
$\btheta_0$. For Model~\ref{A}, the loss of efficiency is in general
of the order of a few percent (the 95\% pointwise range is below 1.08
for $\gamma\geq5$). For smoother processes, the taper range needs to
be increased in order to maintain the same efficiency. This is in sync
with infill-asymptotic results (see, e.g., Figure~3 of
\citealp{Furr:Gent:Nych:06}).  There is little difference between the
Wendland$_1$ and Wendland$_2$ tapers. Overall, the former having in
general a slightly smaller MSPE.

The third row of Figure~\ref{fig:msper} illustrates why it is
prohibitive to use tapers that are linear at the origin. While the
spherical taper has no influence on the screening effect
\citep{Stei:02:SEK} of the exponential Model~\ref{A} (left panel) it
completely breaks down for smoother fields (right panel).

Figure~\ref{fig:msper} also links the taper range with the number of
observations within the taper range. The MSPE ratios suggest that
tapering with more than 100 locations within the taper range is hardly
worth the effort.

In Figure~\ref{fig:msper}, we distinguish a small loss of efficiency
when using taper function~\ref{iv} compared with \ref{i}
and~\ref{ii}. This can be explained by the fact that the taper
function~\ref{iv} does not satisfy Condition~\ref{cond:taper:no:rate}
(as $t_{12}(\0)<1$). Nevertheless, this loss is far less pronounced
than when using taper function~\ref{iii} for model~\ref{B}.

For very small taper ranges, the MSPE ratios shown in
Figure~\ref{fig:msper} seem large. However, presented in terms of
differences, the effect of tapering is hardly noticeable. For example,
for the setting~(Ai) with $n=400$, the MSPEs are 0.1155 0.1101 0.1098
for $\gamma=3,11,\infty$, respectively (see also red line in the left
panel of Figure~\ref{fig:predicitons}).

\begin{figure}[ht]
  \includegraphics[width=.46\textwidth]{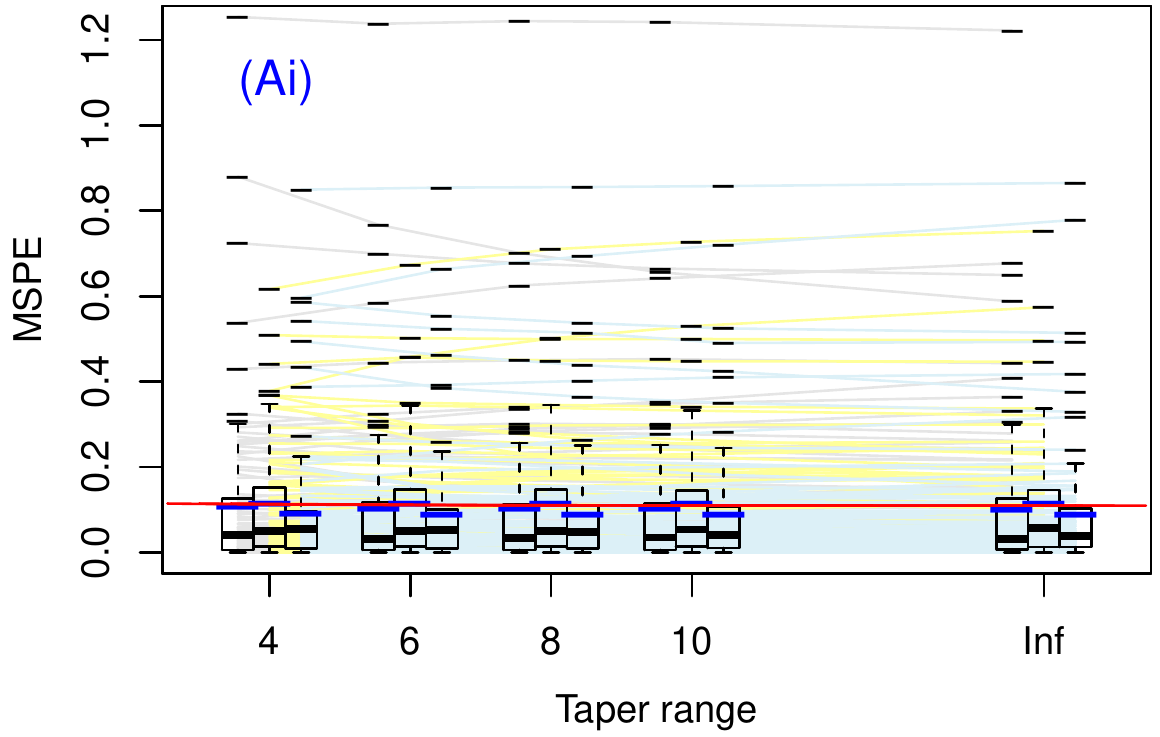}\hfill
\includegraphics[width=.53\textwidth]{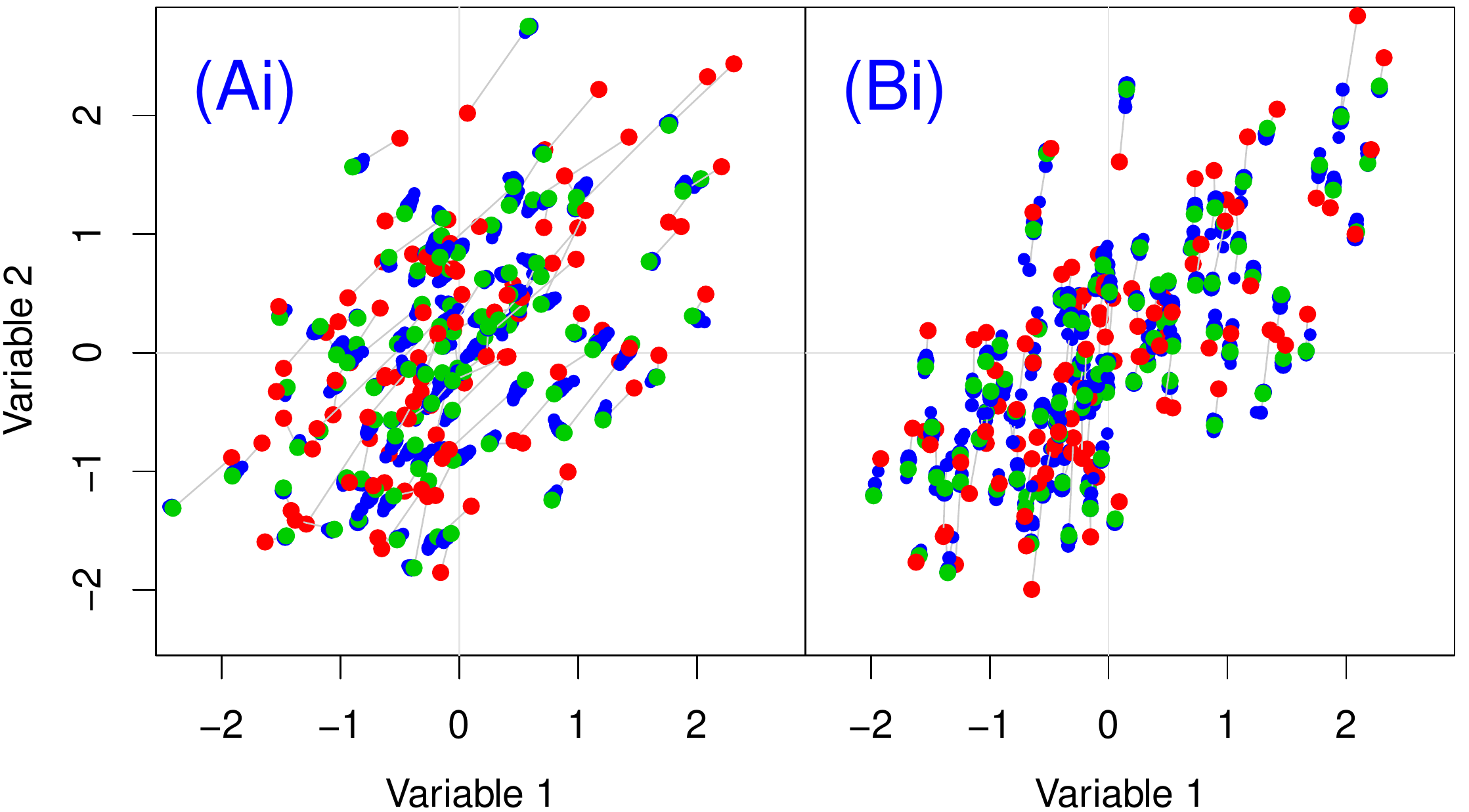}
\caption{Left: Effect of increasing $n$ on the prediction
  error. Horizontal red lines give the theoretical MSPEs. Within each
  boxplot triplet for a specific taper range, left is for $n=400$
  (gray), middle for 1024 (yellow), and right for 2500 (light
  blue). Prediction is based on $\widehat\btheta_\text{tML}$ with
  $\Delta=1$ and 100 realizations of the bivariate process.  Mean is
  indicated by the blue tick. Right: 100 bivariate predictions for
  $n=400$ and $\Delta=1$. Red: simulated ``truth'', green: no
  tapering, blue: tapering with different taper ranges. }
  \label{fig:predicitons}
\end{figure}

The left panel of Figure~\ref{fig:predicitons} further shows the
effect of increasing the number of locations on the MSPE.  The effect
of increasing $n$ is negligible even for the theoretical MSPE, the
values are visually indistinguishable. With as few as $n=400$ we extract
essentially all the information in the system.

The right panel of Figure~\ref{fig:predicitons} shows the results of
100 bivariate predictions at the origin.  There is again virtually no
difference in the predictions using $\gamma=4,6,8,10$ (blue
dots) and no tapering ($\gamma=$~Inf, green dot). For smoother fields
(variable 1, \ref{B}), the prediction error is smaller and thus the
difference between the red and blue/green dots is much smaller than
for variable 2. The choice of the taper matrix function has again only
a marginal effect on the result (not shown).

It has to be kept in mind that our simulation setup is the ``least''
favorable for the tapering approach.  By including a nugget or
reducing the spatial correlation we would receive even more appealing
results because the importance of neighboring locations and their
contribution to the prediction would be less important.  Note also
that estimation and prediction results can be improved by lowering
$\Delta$.

\subsection{Computational efficiency}
The analysis has been implemented with the freely available computer
software R \citep{Ihak:Gent:96,R:15} running on a server with an Intel
Xeon 6C E5-2640 2.50 GHz CPU (24 cores) and $256$GB shared RAM
(parallelization has not been explicitly exploited).  The number of
locations was kept below $2500$ in order to maintain a reasonable computing
time for the untapered settings, which require $\cO(p^3n^3)$ computing
time and $\cO(p^2n^2)$ storage using straightforward R commands with
classical methodologies.

The tapered settings have been implemented using sparse matrix data
structures and algorithms.  The package \rr{spam}
\citep{Furr:15,Furr:Sain:10} is tailored in order to handle tapered
covariance matrices, estimation, and prediction in the framework of
Gaussian random fields.  The core work load consists of calculating a
Cholesky factorization of a permutation of the possibly tapered
covariance matrix. The permutation (multiple minimum degree) improves
storage and operation count; see~\citet{Furr:Sain:10}, \citet{Liu:85},
and \citet{Ng:Peyt:93} for more technical details. From the Cholesky
factor, it is straightforward to calculate the determinant as well as
the quadratic term through two triangular solves. Hence, for large
$n$, there is little difference in computational cost between a
likelihood evaluation or a prediction. Exact operation counts
are difficult to determine but the algorithms are virtually
$\cO(pnh^2)$ for operation count and $\cO(pnh)$ for storage, where $h$
is the ``typical'' number of observations within the taper range.

For estimation, depending on the exact implementation, many likelihood
evaluations are necessary.  Using resonable starting values, the R
function \rr{optim} required on average between $100$ to $250$
function evaluations depending on taper range and model ($n=400$). In
the untapered case, the average was typically somewhat lower.  To
reduce convergence issues, we started estimating the untapered version
using the true parameter values as starting values and subsequently
decreased the taper range using the previous optimum as starting
values. Because of the large size of the datasets, no convergence
issues were encountered and no sample was ``manually'' treated or
eliminated.

\section{Discussion and outlook}\label{sec:out}

Similarly to the univariate case, multivariate tapering is a very
effective approximation approach for prediction and for estimation of
spatially correlated random processes. The small loss in prediction
efficiency is recouped by the computational gains for reasonably 
large data sizes. For very large datasets, approximations have to be
included and tapering is the method of choice as the computational
implementation is straightforward. Compared with other approximation
approaches (low-rank models, e.g.,
\citealp{Cres:Joha:08,Bane:etal:08,Stei:08}, composite likelihood
approaches, e.g., \citealp{Stei:etal:04,Beli:etal:12,Eids:etal:14},
Gaussian Markov random fields type approximations, e.g.,
\citealp{Hart:Hoes:08,Lind:Rue:Lind:11}, etc) tapering is the most
accessible and most scalable approach.

Tapering is especially powerful for prediction. Even for very small
tapers we have a MSPE that is almost identical to the MSPE for the
untapered setting. However, we are substantially faster as a single
prediction is roughly $20$ and $100$ times faster compared with a
classical approach (for $n=2500$ and $n=10000$ using $\gamma=5$).  One
likelihood evaluation is similarly computing intensive as a single
prediction and thus the same advantages hold for estimation.  If the
ultimate goal is prediction, we advocate the use of the one-taper ML
plugin estimates.  The two-taper approach is computationally
self-defeating and should only be used if unbiased estimates are
absolutely necessary.

In the case where the different variables have a similar density of
locations, we propose to use the same taper function for all direct
and cross covariances.  Compared with the taper range, the exact form
of the taper plays a secondary role. Hence for different location
sampling densities, possibly non-stationary, we foresee adaptive
tapers as outlined by \cite{Ande:etal:13} or \cite{Beli:etal:15} as a
valuable alternative.

\bigskip 

For estimation, the standard optimization routines of R (\rr{optim}
and its derivatives) require a substantial amount of time.  We are
currently experimenting with a simple grid search algorithm that
would approximate the ML estimate sufficiently well. Based on the
simulation results in the last section, if prediction based on plugin
estimates is of interest, the approximation is sufficient.

While the uncertainty of the ML estimates can be harnessed
through the Hessian (by product of the \rr{optim} routine)
sufficiently well, deriving uncertainty estimates for an entire
prediction field remains a bottleneck, as accordingly many linear
systems have to be solved.

\section*{Acknowledgments}
RF  acknowledges support of the UZH Research Priority
Program (URPP) on ``Global Change and Biodiversity'' and the Swiss National Science
Foundation SNSF-143282. 
FB presented the content of this paper at the statistics working group of the University of Vienna where he benefited from constructive comments. 

\section*{Appendix}

\subsection*{Proof of the theorems}

\begin{proof}[Proof of Theorem~\ref{thm:consistent:estimation}] 

Because $\Theta$ is compact and because of Lemma~\ref{lem:bounded:derivative}, it is sufficient to show that, for any fixed $\btheta$, $L_{\btheta} - \bar{L}_{\btheta} = o_p(1)$. Hence, let an arbitrary $\btheta$ be fixed. We have
\begin{eqnarray} \label{eq:Tun:Tdeux}
L_{\btheta} - \bar{L}_{\btheta} & = & \frac{1}{np} \log{ \left( \det{ \left[  \bSigma_{\btheta} \K_{\btheta}^{-1} \right] } \right) } + \frac{1}{np} \z \T ( \bSigma_{\btheta}^{-1} - \K_{\btheta}^{-1}  ) \z \nonumber \\
& = & T_1 + T_2.
\end{eqnarray}
We treat $T_1$ and $T_2$ separately. First
\[
T_1 = \frac{1}{np} \sum_{i=1}^{np} \log{ \left( \lambda_i \left[  \K_{\btheta}^{-1/2} \bSigma_{\btheta} \K_{\btheta}^{-1/2} \right] \right)}.
\]
The $\lambda_i(\cdot)$ above are between two constants $0< A$ and $B < + \infty$ uniformly in $i$ and $n$ because of
Condition~\ref{cond:minimal:eigenvalue} and Lemma~\ref{lem:bounded:max:singular:value}. Thus, there exists a finite constant $C$ so that for any $i,n$
\[
\left| \log{ \left( \lambda_i \left[ \K_{\btheta}^{-1/2} \bSigma_{\btheta} \K_{\btheta}^{-1/2} \right] \right)} \right| \leq C
\left| 1 - \lambda_i \left[ \K_{\btheta}^{-1/2}  \bSigma_{\btheta} \K_{\btheta}^{-1/2} \right] \right|.
\]
Thus
\begin{eqnarray*}
| T_1 | & \leq &  \frac{C}{np} \sum_{i=1}^{np} \left| 1 - \lambda_i \left[ \K_{\btheta}^{-1/2} \bSigma_{\btheta} \K_{\btheta}^{-1/2} \right] \right|
 \\
\mbox{(Cauchy-Schwarz:)} & \leq & \frac{C}{np} \sqrt{np} \sqrt{ \sum_{i=1}^{np} \left| 1 - \lambda_i \left[  \K_{\btheta}^{-1/2} \bSigma_{\btheta} \K_{\btheta}^{-1/2} \right] \right|^2 }  \\
& = & C \sqrt{ \frac{1}{np} \tr \left( \left\{ \I - \K_{\btheta}^{-1/2} \bSigma_{\btheta} \K_{\btheta}^{-1/2} \right\}^2 \right) } \\
& = & C \sqrt{ \frac{1}{np} \tr \left( \left\{ \K_{\btheta}^{-\frac{1}{2}} \left[  \K_{\btheta} - \bSigma_{\btheta} \right] \K_{\btheta}^{-\frac{1}{2}} \right\}^2  \right) } \\
& = & C \sqrt{ \frac{1}{np} \left| \left| \K_{\btheta}^{-\frac{1}{2}} \left[  \K_{\btheta} - \bSigma_{\btheta} \right] \K_{\btheta}^{-\frac{1}{2}} \right| \right|_F^2 }.
\end{eqnarray*}
Now, because of Condition~\ref{cond:minimal:eigenvalue}, $\rho_1(  \K_{\btheta}^{-\frac{1}{2}} )$ is bounded uniformly in $n$ by a finite constant $D$. Hence we have
\[
|T_1| \leq CD^2 \sqrt{ \frac{1}{np} \left| \left|   \K_{\btheta} - \bSigma_{\btheta}  \right| \right|_F^2 },
\]
which goes to $0$ as $n \to \infty$ because of Lemma~\ref{lem:difference:frobenius:norm}.
Next, turning to $T_2$ in \eqref{eq:Tun:Tdeux},
\begin{eqnarray*}
\E \left(  T_2  \right) & = & 
\frac{1}{np} \tr \left( \bSigma_{\btheta_0} \left( \bSigma_{\btheta}^{-1} - \K_{\btheta}^{-1}  \right) \right) \\
& = & \frac{1}{np} \tr \left( \bSigma_{\btheta_0} \K_{\btheta}^{-1} \left( \K_{\btheta} - \bSigma_{\btheta}  \right) \bSigma_{\btheta}^{-1} \right).
\end{eqnarray*}
Hence, interpreting $\tr \left( \A \B \right)$ as a scalar product between $\A$ and $\B \T$, we obtain by the Cauchy-Schwarz inequality
\[
\left| \E \left(  T_2  \right) \right| \leq \sqrt{ \frac{1}{np} || \bSigma_{\btheta}^{-1} \bSigma_{\btheta_0} \K_{\btheta}^{-1} ||_F^2 } \sqrt{ \frac{1}{np} || \K_{\btheta} - \bSigma_{\btheta}  ||_F^2 }.
\]
In the above display, the first square root is bounded because of Condition~\ref{cond:minimal:eigenvalue} and of Lemma~\ref{lem:bounded:max:singular:value}. The second square root goes to $0$ because of Lemma~\ref{lem:difference:frobenius:norm}. Hence $\E (T_2) \to_{n \to \infty} 0$. Furthermore
\begin{eqnarray*}
\var \left( T_2 \right) & = & \frac{2}{(np)^2} \tr \left(  \bSigma_{\btheta_0} \left[ \bSigma_{\btheta}^{-1} - \K_{\btheta}^{-1} \right] \bSigma_{\btheta_0} \left[ \bSigma_{\btheta}^{-1} - \K_{\btheta}^{-1} \right] \right)
\\
& \leq & \frac{2}{np} \rho_1(\bSigma_{\btheta_0})^2 \left[ \rho_1( \bSigma_{\btheta}^{-1} ) + \rho_1( \K_{\btheta}^{-1} ) \right]^2 .
\end{eqnarray*}
In the above display, the $\rho_1(\cdot)$ are bounded because of Condition~\ref{cond:minimal:eigenvalue} and Lemma~\ref{lem:bounded:max:singular:value}. Thus $\var(T_2) \to_{n \to \infty} 0$. So $T_2 = o_p(1)$ which finishes the proof.
\end{proof}

\begin{proof}[Proof of Theorem~\ref{thm:convergence:IMSE}]
We only prove \eqref{eq:convergence:IMSE}, the proof of \eqref{eq:convergence:fixed:MSE} being similar and technically simpler.
Using $a^2 - b^2 = (a+b)(a-b)$ followed by the Cauchy-Schwarz inequality, we obtain
\begin{flalign} \label{eq:Uun:Udeux}
& \sup_{\btheta}
\left| 
\int_{\cD_n} \left[ \bsigma_{\btheta}(\x) \T \bSigma_{\btheta}^{-1} \z - Z_1(\x) \right]^2  f_n(\x) d \x
-
\int_{\cD_n} \left[ \k_{\btheta}(\x) \T \K_{\btheta}^{-1} \z - Z_1(\x) \right]^2  f_n(\x) d \x
\right| & \nonumber \\
& \leq 
\int_{\cD_n}
\sup_{\btheta} \left( \left|  \bsigma_{\btheta}(\x) \T \bSigma_{\btheta}^{-1} \z + \k_{\btheta}(\x) \T \K_{\btheta}^{-1} \z -2  Z_1(\x) \right|
\left|  \bsigma_{\btheta}(\x) \T \bSigma_{\btheta}^{-1} \z - \k_{\btheta}(\x) \T \K_{\btheta}^{-1} \z  \right|
\right)  f_n(\x) d \x & \nonumber \\
  & \leq 
\sqrt{  \int_{\cD_n}
\sup_{\btheta} \left( \bsigma_{\btheta}(\x) \T \bSigma_{\btheta}^{-1} \z + \k_{\btheta}(\x) \T \K_{\btheta}^{-1} \z -2  Z_1(\x) \right)^2
  f_n(\x) d \x } & \nonumber \\
&  ~ ~ ~ \sqrt{ \int_{\cD_n}
\sup_{\btheta} 
\left(  \bsigma_{\btheta}(\x) \T \bSigma_{\btheta}^{-1} \z - \k_{\btheta}(\x) \T \K_{\btheta}^{-1} \z  \right)^2
  f_n(\x) d \x } & \nonumber \\
  & =  \sqrt{U_1} \sqrt{U_2}. &
\end{flalign}

We show separately that $U_1 = O_p(1)$ and $U_2 = o_p(1)$. For $U_1$,
\begin{eqnarray*}
U_1  & \leq & 3 \int_{\cD_n}
\sup_{\btheta} \left( \bsigma_{\btheta}(\x) \T \bSigma_{\btheta}^{-1} \z  \right)^2
  f_n(\x) d \x
  +
 3 \int_{\cD_n}
\sup_{\btheta} \left(   \k_{\btheta}(\x) \T \K_{\btheta}^{-1} \z  \right)^2
  f_n(\x) d \x \\
& & +12 \int_{\cD_n}
\sup_{\btheta} \left(   Z_1(\x) \right)^2
  f_n(\x) d \x.
\end{eqnarray*}

The last random integral in the above display has constant mean value $12 c_{11}(\0;\btheta_0)$ so it is bounded in probability. We address the two remaining random integrals in the same way, and give the details for the first one only. Using a version of Sobolev embedding theorem (Theorem 4.12, Part I, Case A in \citealp{adams03sobolev}), there exists a finite constant $A_{\Theta}$ depending only on $\bTheta$ so that
\begin{eqnarray*}
\sup_{\btheta}  \left(   \bsigma_{\btheta}(\x) \T \bSigma_{\btheta}^{-1} \z  \right)^2 
  \leq A_{\Theta} \int_{\Theta} \left| \left(   \bsigma_{\btheta}(\x) \T \bSigma_{\btheta}^{-1} \z  \right)^2 \right|^{q+1} d \btheta
  + 
A_{\Theta}  \sum_{i=1}^q \int_{\Theta}  \left| \frac{\partial}{ \partial \theta_i} \left[ \left( \bsigma_{\btheta}(\x) \T \bSigma_{\btheta}^{-1} \z   \right)^2 \right] \right|^{q+1} d \btheta.
\end{eqnarray*}
Hence, using Fubini theorem for non-negative integrand and $(|a|+|b|)^{q+1} \leq 2^{q+1} (|a|^{q+1} + |b|^{q+1})$, we obtain
\begin{flalign*}
& \E \left( \int_{\cD_n}
\sup_{\btheta} \left( \bsigma_{\btheta}(\x) \T \bSigma_{\btheta}^{-1} \z  \right)^2
  f_n(\x) d \x \right) &  \\ 
 & \leq A_{\Theta} \int_{\Theta} \int_{\cD_n} \E \left( \left| \left(   \bsigma_{\btheta}(\x) \T \bSigma_{\btheta}^{-1} \z  \right)^2 \right|^{q+1} \right)  f_n(\x) d \x d \btheta & \\
& ~ ~ ~ ~ +  A_{\Theta} 2^{2q+2} \sum_{i=1}^q \int_{\Theta} \int_{\cD_n} \E \left( \left| \left( \frac{\partial \bsigma_{\btheta}(\x)\T}{ \partial \theta_i}  \bSigma_{\btheta}^{-1} \z  \right) \left( \bsigma_{\btheta}(\x)\T \bSigma_{\btheta}^{-1} \z   \right) \right|^{q+1} \right) f_n(\x) d \x d \btheta  & \\
& ~ ~ ~ ~ + A_{\Theta} 2^{2q+2} \sum_{i=1}^q \int_{\Theta} \int_{\cD_n} \E \left( \left| \left( \bsigma_{\btheta}(\x)\T
\bSigma_{\btheta}^{-1} \frac{\partial \bSigma_{\btheta}}{ \partial \theta_i} \bSigma_{\btheta}^{-1} \z  \right) \left( \bsigma_{\btheta}(\x) \T \bSigma_{\btheta}^{-1} \z   \right) \right|^{q+1} \right) f_n(\x) d \x d \btheta. &
\end{flalign*}
Let $\lambda(\Theta)$ be the Lebesgue measure of $\Theta$. Using the Cauchy--Schwarz inequality and letting $B_{q+1}$ be the positive constant so that, for $X$ following a Gaussian distribution with zero mean, $\E(X^{2(q+1)}) = B_{q+1} ( \E(X^2) )^{q+1}$, we obtain, by letting $D = A_{\Theta} B_{q+1} \lambda(\Theta) 2^{2q+2}$,
\begin{flalign} \label{eq:after:B:qplus1}
& \E \left( \int_{\cD_n}
\sup_{\btheta} \left( \bsigma_{\btheta}(\x)\T \bSigma_{\btheta}^{-1} \z  \right)^2
  f_n(\x) d \x \right) &  \\ 
 & \leq A_{\Theta} B_{q+1} \lambda(\Theta) \sup_{\x,\btheta} \E^{q+1} \left(  \left(   \bsigma_{\btheta}(\x)\T \bSigma_{\btheta}^{-1} \z  \right)^2  \right)   & \nonumber \\
&~ ~ ~ ~ +  D \sum_{i=1}^q \sup_{\x,\btheta} \sqrt{ \E^{q+1} \left(  \Big( \frac{\partial \bsigma_{\btheta}(\x)\T}{ \partial \theta_i}  \bSigma_{\btheta}^{-1} \z  \Big)^2 \right) }
  \sup_{\x,\btheta} \sqrt{ \E^{q+1} \left(  \left( \bsigma_{\btheta}(\x)\T \bSigma_{\btheta}^{-1} \z   \right)^2 \right)  } & \nonumber \\
& ~ ~ ~ ~ + D \sum_{i=1}^q \sup_{\x,\btheta} \sqrt{ \E^{q+1}  \left( \Big( \bsigma_{\btheta}(\x)\T
\bSigma_{\btheta}^{-1} \frac{\partial \bSigma_{\btheta}}{ \partial \theta_i} \bSigma_{\btheta}^{-1} \z \Big)^2 \right)   }
  \sup_{\x,\btheta} \sqrt{ \E^{q+1} \left(  \left( \bsigma_{\btheta}(\x)\T \bSigma_{\btheta}^{-1} \z   \right)^2 \right) }. & \nonumber
\end{flalign}

Now, all the $\E^{q+1}(\cdot)$ above are of the form $\E^{q+1}( [\w_{\btheta}(\x) \T \bfM_{\btheta} \z ]^2 )$.
Furthermore, $\bfM_{\btheta}$ is symmetric and satisfies, by using Condition~\ref{cond:minimal:eigenvalue} and Lemma~\ref{lem:bounded:max:singular:value}, $\sup_{\btheta} \rho_1( \bfM_{\btheta} ) \leq C$ for a finite constant $C$.
Finally, for $i=k(n-1)+a$, with $k=1,\dots,p$ and $a=1,\dots,n$, $\sup_{\btheta} |\w_{\btheta}(\x)_{i}| \leq G / (1+|\x - \x_a|^{d+\alpha})$, for a finite constant $G$. Hence,
\begin{eqnarray*}
\sup_{\x,\btheta} \E( [ \w_{\btheta}(\x) \T \bfM_{\btheta} \z ]^2 ) & = & \sup_{\x,\btheta} \w_{\btheta}(\x) \T \bfM_{\btheta} \bSigma_{\btheta_0} \bfM_{\btheta} \w_{\btheta}(\x) \\
& \leq & \sup_{\x,\btheta} ||\w_{\btheta}(\x) ||^2 C^2 \sup_{\btheta} \rho_1(\bSigma_{\btheta_0}),
\end{eqnarray*}
which is bounded because of Lemmas~\ref{lem:bounded:sum} and~\ref{lem:bounded:max:singular:value}. 
Hence, in \eqref{eq:Uun:Udeux}, $U_1 = O_p(1)$. Let us now turn to $U_2$. Using the Sobolev embedding theorem again with the constant $A_{\Theta}$, we obtain
\begin{eqnarray*}
\E( U_2 ) & \leq & A_{\Theta}   \int_{\Theta} \int_{\cD_n} \E \left( \left| \left[ \bsigma_{\btheta}(\x) \T \bSigma_{\btheta}^{-1} \z -  \k_{\btheta}(\x) \T \K_{\btheta}^{-1} \z \right]^2 \right|^{q+1} \right) f_n(\x) d \x  d \btheta \\
& & + A_{\Theta} \sum_{i=1}^q  \int_{\Theta} \int_{\cD_n} \E \left(  \left| \frac{\partial}{\partial \theta_i} \left( \left[ \bsigma_{\btheta}(\x) \T \bSigma_{\btheta}^{-1} \z -  \k_{\btheta}(\x) \T \K_{\btheta}^{-1} \z \right]^2 \right) \right|^{q+1} \right) f_n(\x) d \x  d \btheta \\
& = & A_{\Theta} I_0 + A_{\Theta} \sum_{i=1}^q I_i
\end{eqnarray*}
In the above display, we only show that the integrals $I_1,\dots,I_q$ converge to $0$, since it is more difficult than for the integral $I_0$. Hence let us fix an integer $i$ in $\{1,\dots,q\}$. Using Cauchy-Schwarz inequality, we have
\begin{eqnarray*}
I_i & \leq & A_{\Theta} \lambda(\Theta) 2^{q+1} \sup_{\x,\btheta} \sqrt{ \E \left( \Big|
 \frac{\partial}{\partial \theta_i} \left[ \bsigma_{\btheta}(\x) \T \bSigma_{\btheta}^{-1} \z -  \k_{\btheta}(\x) \T \K_{\btheta}^{-1} \z \right]  \Big|^{2(q+1)} \right) } \\
 & & \times~
 \sup_{\x,\btheta} \sqrt{ \E \left( \Big|
 \bsigma_{\btheta}(\x) \T \bSigma_{\btheta}^{-1} \z -  \k_{\btheta}(\x) \T \K_{\btheta}^{-1} \z  \Big|^{2(q+1)} \right) }. 
\end{eqnarray*}
Again, both of the supremums of square roots in the above display go to $0$ as $n \to \infty$ and we show it only for the first one, since it is more difficult than for the second one. Using the positive constant $B_{q+1}$ used before \eqref{eq:after:B:qplus1}, it is sufficient to show that 
\[
\sup_{\btheta,\x}\E \left( \left\{ \frac{\partial}{\partial \theta_i} \left[ \bsigma_{\btheta}(\x) \T \bSigma_{\btheta}^{-1} \z -  \k_{\btheta}(\x) \T \K_{\btheta}^{-1} \z \right]  \right\}^2 \right) 
\]
goes to $0$ as $n \to \infty$. Then, we use
\[
(a_{11} - a_{22} )^2 \leq 2 \left[ (a_{11} - a_{21})^2 + (a_{21} - a_{22})^2 \right]
\]
and
\[
(b_{1111} - b_{2222})^2 \leq 4 \left[ (b_{1111} - b_{2111})^2 + (b_{2111} - b_{2211})^2 + (b_{2211} - b_{2221})^2+ (b_{2221} - b_{2222})^2 \right],
\] 
where subscripts $1$ and $2$ denote ``untapered'' and ``tapered'' and where for example $a_{21} = \{[\partial \k_{\btheta}(\x)]/[\partial \theta_i]\} \T \bSigma_{\btheta}^{-1} \z$ and
$b_{2211} = \k_{\btheta}(\x) \T \K_{\btheta}^{-1}  \{ [\partial \bSigma_{\btheta}]/[\partial \theta_i]\} \bSigma_{\btheta}^{-1} \z$. From this, it is sufficient to show that a generic term of the form
\begin{equation} \label{eq:generic:un}
\sup_{\btheta,\x} \E \left( \left[ (\v_{\btheta}(\x) - \w_{\btheta}(\x)) \T \bfM_{\btheta} \z \right]^2 \right),
\end{equation}
\begin{equation} \label{eq:generic:deux}
\sup_{\btheta,\x} \E \left( \left[ \m_{\btheta}(\x)  \T \bfM_{\btheta} ( \bSigma_{\btheta}^{-1} -  \K_{\btheta}^{-1} ) \N_{\btheta} \z \right]^2 \right)
\end{equation}
or
\begin{equation} \label{eq:generic:trois}
\sup_{\btheta,\x} \E \left( \left[ \m_{\btheta}(\x)  \T \bfM_{\btheta} \left( \frac{ \partial \bSigma_{\btheta}}{\partial \theta_i} -  \frac{\partial \K_{\btheta}}{\partial \theta_i} \right) \N_{\btheta} \z \right]^2 \right),
\end{equation}
goes to $0$. In \eqref{eq:generic:un}, \eqref{eq:generic:deux} and \eqref{eq:generic:trois}, $\sup_{\btheta} \rho_1( \bfM_{\btheta} )$ and $\sup_{\btheta} \rho_1( \N_{\btheta} )$ are bounded (Condition~\ref{cond:minimal:eigenvalue} and Lemma~\ref{lem:bounded:max:singular:value}); $\v_{\btheta}(\x) - \w_{\btheta}(\x) =\bsigma_{\btheta}(\x)- \k_{\btheta}(\x)$  or 
$\v_{\btheta}(\x) - \w_{\btheta}(\x) = (\partial \bsigma_{\btheta}(\x))/( \partial \theta_i ) - (\partial \k_{\btheta}(\x))/( \partial \theta_i )$; and $\m_{\btheta}(\x) = \k_{\btheta}(\x)$ or $\m_{\btheta}(\x) = \{[\partial \k_{\btheta}(\x)]/[\partial \theta_i]\}$. 

Let us now show that a generic term of the form \eqref{eq:generic:un} goes to $0$. We have
\begin{eqnarray*}
\sup_{\btheta,\x} \E \left( \left[ (\v_{\btheta}(\x) - \w_{\btheta}(\x)) \T \bfM_{\btheta} \z \right]^2 \right)
& = & \sup_{\btheta,\x}  (\v_{\btheta}(\x) - \w_{\btheta}(\x)) \T \bfM_{\btheta} \bSigma_{\btheta_0} \bfM_{\btheta} \T
(\v_{\btheta}(\x) - \w_{\btheta}(\x)) \\
& \leq & \sup_{\btheta} \rho_1(\bfM_{\btheta} \bSigma_{\btheta_0} \bfM_{\btheta} \T)
\sup_{\btheta,\x} || \v_{\btheta}(\x) - \w_{\btheta}(\x) ||^2,
\end{eqnarray*}
which goes to $0$ as $n \to \infty$ by remembering that $\sup_{\btheta} \rho_1( \bfM_{\btheta} )$ is bounded and by using Lemmas~\ref{lem:bounded:max:singular:value} and~\ref{lem:taper:convergence:sum}.

For a generic term of the form \eqref{eq:generic:deux}, we have
\begin{flalign} \label{eq:for:generic:deux}
& \sup_{\btheta,\x} \E \left( \left[ \m_{\btheta}(\x)  \T \bfM_{\btheta} \left(  \bSigma_{\btheta}^{-1} - \K_{\btheta}^{-1} \right) \N_{\btheta} \z \right]^2 \right) \nonumber & \\
&~ = \sup_{\btheta,\x} \E \left( \left[ \m_{\btheta}(\x)  \T \bfM_{\btheta} \K_{\btheta}^{-1} \left( \K_{\btheta} - \bSigma_{\btheta} \right) \bSigma_{\btheta}^{-1} \N_{\btheta} \z \right]^2 \right)  & \\
&~ = 
\sup_{\btheta,\x} \m_{\btheta}(\x) \T  \bfM_{\btheta} \K_{\btheta}^{-1} \left( \K_{\btheta} - \bSigma_{\btheta} \right) \bSigma_{\btheta}^{-1} \N_{\btheta}
\bSigma_{\btheta_0}
\N_{\btheta} \T
\bSigma_{\btheta}^{-1} \left( \K_{\btheta} - \bSigma_{\btheta} \right) \K_{\btheta}^{-1} \bfM_{\btheta} \T \m_{\btheta}(\x)
& \nonumber \\
&~ \leq  \sup_{\btheta,\x} ||\m_{\btheta}(\x)||^2 \rho_1(\bfM_{\btheta})^2 \rho_1(\N_{\btheta})^2
\rho_1(\bSigma_{\btheta}^{-1})^2 \rho_1(\K_{\btheta}^{-1})^2
\rho_1(\bSigma_{\btheta_0}) \rho_1( \K_{\btheta} - \bSigma_{\btheta} )^2. \nonumber 
\end{flalign}
In the above display, $\sup_{\btheta,\x} ||\m_{\btheta}(\x)||^2$ is bounded because of Lemma~\ref{lem:bounded:sum}. Furthermore  all the $\rho_1(\cdot)^2$, except the last one are bounded uniformly in $\btheta$,
by remembering that $\sup_{\btheta} \rho_1( \bfM_{\btheta} )$ and $\sup_{\btheta} \rho_1( \N_{\btheta} )$ are bounded, and because of Condition~\ref{cond:minimal:eigenvalue} and Lemma~\ref{lem:bounded:max:singular:value}.
Finally $\sup_{\btheta} \rho_1(\K_{\btheta} - \bSigma_{\btheta})$ goes to $0$ as $n \to \infty$ because of Lemma~\ref{lem:vanishing:max:singular:value}. Hence a generic term of the form \eqref{eq:generic:deux} goes to $0$ as $n \to \infty$. Finally, by the same arguments as following \eqref{eq:for:generic:deux}, we show that a generic term of the form \eqref{eq:generic:trois} goes to $0$ as $n \to \infty$. Hence, $\E(U_2)$ in \eqref{eq:Uun:Udeux} goes to $0$ as $n \to \infty$ which concludes the proof.
\end{proof}

\subsection*{Technical results}

The following lemma is a generalization of Lemma D.1 in \cite{bachoc14asymptotic}.

\begin{lem} \label{lem:bounded:sum}
Let $\Delta >0$ and $\alpha >0$ be fixed.
Let $f(\x;\btheta)$ be a family of functions: $\IR^d \to \IR$ so that for all $\btheta \in \Theta$, $|f(\x;\btheta)| \leq 1/(1+|\x|^{d+\alpha})$. Then, for any $m \in \IN^+$, $\v \in \IR^d$, $\s_1,..,\s_m \in \IR^d$, so that for any $i \neq j$
$|\s_i - \s_j| \geq \Delta$, we have
\[
\sup_{\btheta} \sum_{i=1}^m |f(\s_i-\v ;\btheta)| \leq \frac{d2^{2d}}{\Delta^d} \sum_{k=1}^{+ \infty} \frac{k^{d-1}}{1+(k-1)^{d+\alpha}},
\]
where the right-hand term in the above display is a finite constant depending only on $d$, $\Delta$ and $\alpha$.

\end{lem}

\begin{proof}[Proof of Lemma~\ref{lem:bounded:sum}]

By assumption on $f(\x,\btheta)$ we have
\[
\sup_{\btheta} \sum_{i=1}^m |f(\s_i-\v ;\btheta)| \leq
\sum_{i=1}^m \frac{1}{1+|\s_i-\v|^{d+ \alpha}}.
\]
Let, for $k \geq 1$, $N_k$ be the number of points $\s_j$ in $E_k = \{ \w ;  |\w-\v| \leq k \} \backslash \{ \w ;  |\w-\v| \leq k-1 \}$. Then, to the $N_k$ points $\s_j$ that are in $E_k$ we can associate $N_k$ disjoint $|\cdot|$-balls in $E_k$ so that each of them has volume $(\Delta/2)^d$ (recall
$|\a|=\max_l |a_l|$). The total volume occupied by these balls is $N_k(\Delta/2)^d$. On the other hand, the volume of $E_k$ is
\[
(2k)^d - (2k-2)^d = 2^d \int_{k-1}^k d u^{d-1} du \leq 2^d d k^{d-1}.
\]
So we have $N_k \leq d 2^{2d} k^{d-1} / \Delta^d$. The result is then obtained by noting that for $\s_j \in E_k$, $|\s_j - \v| \geq k-1$.
\end{proof}

The following lemma is a generalization of Lemma D.3 in \cite{bachoc14asymptotic}.

\begin{lem} \label{lem:vanishing:rest}
Consider the setting of Lemma~\ref{lem:bounded:sum}.
Then, for any $N \in \IN^+$, for any $m \in \IN^+$, $\v \in \IR^d$, $\s_1,..,\s_m \in \IR^d$, so that for any $i \neq j$
$|\s_i - \s_j| \geq \Delta$, we have
\[
\sup_{\btheta} \sum_{i=1,\dots,m; |\s_i - \v| > N-1 } |f(\s_i-\v ;\btheta)| \leq \frac{d2^{2d}}{\Delta^d} \sum_{k=N}^{+ \infty} \frac{k^{d-1}}{1+(k-1)^{d+\alpha}},
\]
where the right-hand term in the above display is a function of $N$, $d$, $\Delta$ and $\alpha$ only, that goes to $0$ as $N \to + \infty$ and for fixed $d,\Delta,\alpha$.

\end{lem}

\begin{proof}[Proof of Lemma~\ref{lem:vanishing:rest}]
The lemma is obtained by the proof of Lemma~\ref{lem:bounded:sum}, by noting that only the points $\s_j$ that are in $E_k$ for $k \geq N$ give a non-zero contribution to the sum in the left-hand side of the display in the lemma.
\end{proof}

\begin{lem} \label{lem:bounded:max:singular:value}
Assume that Condition~\ref{cond:minimal:distance} holds.
Let $f_{kl}(\x;\btheta)$, $k,l=1,\dots,p$ be $p^2$ functions: $\IR^d \to \IR$ so that for all $\btheta \in \Theta$, $|f_{kl}(\x;\btheta)| \leq 1/(1+|\x|^{d+\alpha})$ and $f_{kl}(\x;\btheta)=f_{lk}(-\x;\btheta)$.
Let $\F_{\btheta}$ be the $np \times np$ matrix defined by, for $i = (k-1)n + a$ and $j = (l-1) n + b$, with $k,l=1,\dots,p$ and $a,b=1,\dots,n$, $f_{\btheta i j} = f_{kl}(\x_a-\x_b ; \btheta)$. 
Then, there exists a constant $A < \infty$ so that for any $n$, $\btheta$, $\rho_1(\F_{\btheta}) \leq A$.
\end{lem}

\begin{proof}[Proof of Lemma~\ref{lem:bounded:max:singular:value}]

Since $\F_{\btheta}$ is symmetric, $\rho_1(\F_{\btheta}) = \lambda_1(\F_{\btheta})$. Hence, because of Gershgorin circle theorem and of $| f_{\btheta k k} | \leq 1$ for any $n,\btheta$, it is sufficient to show that 
\[
\sup_{i,n,\btheta} \sum_{j=1,\dots,np; j \neq i} |f_{\btheta i j}| 
\]
is finite. By writing the sum above as the sum of $p$ subsums, it is sufficient to show that
\[
\sup_{k,l,a,n,\btheta} \sum_{j=1,\dots,n} |f_{kl} (\x_a - \x_j ; \btheta)| 
\]
is finite. This is true because of Lemma~\ref{lem:bounded:sum}.
\end{proof}

\begin{lem} \label{lem:bounded:derivative}
Assume that conditions~\ref{cond:ctheta:one:derivative},~\ref{cond:minimal:distance}, and~\ref{cond:minimal:eigenvalue} hold.
Then, as $n \to \infty$
\[
\sup_{i,\btheta} \left| \frac{\partial}{\partial \theta_i} L_{\btheta} \right| = O_p(1)
~ ~ ~ ~ ~
\mbox{and}
~ ~ ~ ~ ~
\sup_{i,\btheta } \left| \frac{\partial}{\partial \theta_i} \bar{L}_{\btheta} \right| = O_p(1).
\]

\end{lem}

\begin{proof}[Proof of Lemma~\ref{lem:bounded:derivative}]
We do the proof for $L_{\btheta}$ only since the proof for $\bar{L}_{\btheta}$ is identical. We have for any $i=1,\dots,q$,
\begin{eqnarray*}
\sup_{\btheta \in \Theta} \left| \frac{\partial}{\partial \theta_i} L_{\btheta} \right|
& = & \sup_{\btheta \in \Theta} \left| \frac{1}{np} \tr \left( \bSigma_{\btheta}^{-1} \frac{\partial \bSigma_{\btheta}}{\partial \theta_i} \right) - \frac{1}{np} \z \T \bSigma_{\btheta}^{-1} \frac{\partial \bSigma_{\btheta}}{\partial \theta_i} \bSigma_{\btheta}^{-1} \z \right| \\
& \leq &  \sup_{\btheta} \rho_1\left( \bSigma_{\btheta}^{-1/2} \frac{\partial \bSigma_{\btheta}}{\partial \theta_i} \bSigma_{\btheta}^{-1/2} \right) + \frac{1}{np} \z \T \z \sup_{\btheta} \rho_1\left( \bSigma_{\btheta}^{-1} \frac{\partial \bSigma_{\btheta}}{\partial \theta_i} \bSigma_{\btheta}^{-1} \right).
\end{eqnarray*}

Now, $(1/(np)) \z^T \z$ is bounded in probability since it is positive with constant mean value $ (1/p) \sum_{k=1}^p c_{kk}(\0;\btheta_0)$. The two $\rho_1(\cdot)$ in the above display are bounded uniformly in $\btheta$ because of $\rho_1(\C \D) \leq \rho_1(\C) \rho_1(\D)$, of Conditions~\ref{cond:ctheta:one:derivative}, \ref{cond:minimal:distance}, and~\ref{cond:minimal:eigenvalue} and of Lemma~\ref{lem:bounded:max:singular:value}.
\end{proof}

\begin{lem} \label{lem:taper:convergence:sum}
Let $\alpha>0$ and $\Delta>0$ be fixed.
Let $f(\x;\btheta)$ be a family of functions: $\IR^d \to \IR$ so that for all $\btheta$, $|f(\x;\btheta)| \leq 1/(1+|\x|^{d+ \alpha})$. Let $t(\x)$ be a fixed function: $\IR^d \to \IR$ that is continuous at $\0$ and so that $t(\0) = 1$ and $|t(\x)| \leq 1$. Let $S_m$ be the set of all sets of points $(\s_1,\dots,\s_m)$ so that for $i \neq j$ $|\s_i - \s_j| \geq \Delta$. Then,
\[
\sup_{m,(\s_1,\dots,\s_m) \in S_m,\v,\btheta} \, \, \, \, \sum_{i=1}^m \left| f(\v-\s_i ; \btheta) - f(\v-\s_i ; \btheta) t\big( (\v-\s_i)/\gamma\big) \right|
\]
goes to $0$ as $\gamma \to \infty$.

\end{lem}

\begin{proof}[Proof of Lemma~\ref{lem:taper:convergence:sum}]
Let $\epsilon >0$ be fixed. Because of Lemma~\ref{lem:vanishing:rest}, we can find $M \in \IN^+$ so that 
\[
\sup_{m,(\s_1,\dots,\s_m) \in S_m,\v,\btheta} \, \, \, \, \sum_{i=1,\dots,m;|\v-\s_i| > M-1} \left| f(\v-\s_i ; \btheta) - f(\v-\s_i ; \btheta) t\big((\v-\s_i)/\gamma\big) \right| \leq \epsilon.
\]
Because $t$ is continuous at $\0$, we have for $\gamma$ large enough  and for $|\v - \s_i| \leq M-1$
\[
\left| 1 - t \big( (\s_i-\v) /\gamma\big)\right| \leq \frac{\epsilon}{\tilde{N}_{M-1}},
\]
where $\tilde{N}_{M-1}$ is the maximum numbers of points $\s_j$ so that $|\s_j-\v| \leq M-1$, over all possible $m$, $\v$ and
$(\s_1,\dots,\s_m) \in S_m$. Putting the two bounds together, and using $|f(\x;\btheta)| \leq 1$ we obtain, for $\gamma$ large enough,
\[
\sup_{m,(\s_1,\dots,\s_m) \in S_m,\btheta} \, \, \, \, \sum_{i=1}^m \left| f(\v-\s_i ; \btheta) - f(\v-\s_i ; \btheta) t\big( (\v-\s_i)/\gamma\big) \right|
\leq \epsilon + \tilde{N}_{M-1}\frac{\epsilon}{\tilde{N}_{M-1}},
\]
which finishes the proof.
\end{proof}

\begin{lem} \label{lem:vanishing:max:singular:value}
Assume that Conditions~\ref{cond:taper:no:rate} and~\ref{cond:minimal:distance} hold.
Let $f_{kl}(\x;\btheta)$ and $\F_{\btheta}$ be as in Lemma~\ref{lem:bounded:max:singular:value}. Let $t_{kl}(\x)$, $k,l=1,\dots,p$, be the $p^2$ taper functions satisfying Condition~\ref{cond:taper:no:rate}.
Let $\gamma$ be the taper range, also satisfying Condition~\ref{cond:taper:no:rate}.
Let $\G_{\btheta}$ be the $np \times np$ matrix defined by, for $i = (k-1)n + a$ and $j = (l-1) n + b$, with $k,l=1,\dots,p$ and $a,b=1,\dots,n$, $g_{\btheta i j} = f_{kl}(\x_a-\x_b ; \btheta) t_{kl}\big( (\x_a- \x_b ) / \gamma \big) $.
Then, $\sup_{\btheta} \rho_1(\F_{\btheta}  - \G_{\btheta}) \to_{n \to \infty} 0$.
\end{lem}

\begin{proof}[Proof of Lemma~\ref{lem:vanishing:max:singular:value}]

The lemma is a consequence of Lemma~\ref{lem:taper:convergence:sum}. The proof is based on Gershgorin circle theorem as for the proof of Lemma~\ref{lem:bounded:max:singular:value}.
\end{proof}

\begin{lem} \label{lem:difference:frobenius:norm}
Assume that Conditions~\ref{cond:ctheta:one:derivative}, \ref{cond:taper:no:rate}, and~\ref{cond:minimal:distance} hold. Then, $\sup_{\btheta} \frac{1}{np} || \bSigma_{\btheta} - \K_{\btheta} ||_F^2$ goes to $0$ as $n \to \infty$.
\end{lem}

\begin{proof}[Proof of Lemma~\ref{lem:difference:frobenius:norm}]

The lemma is a consequence of Lemma~\ref{lem:vanishing:max:singular:value}.
\end{proof}

\renewcommand{\baselinestretch}{1.1}\rm

\end{document}